\documentclass[12pt,reqno,english]{smfart}

\usepackage{amsmath,amssymb,smfthm,stmaryrd,epic,enumerate,dsfont}
\usepackage[english]{babel}
\usepackage{xcolor}
\usepackage[all]{xy}
\usepackage{csquotes}
\usepackage{mathabx}

\usepackage{mathrsfs}

\newcommand*{\htarrow}{\lhook\joinrel\twoheadrightarrow}

\NumberTheoremsIn{subsection}

\SwapTheoremNumbers

\begin{document}

\newcommand{\REMARK}[1]{\marginpar{\tiny #1}}
\newtheorem{thm}{Theorem}[subsection]
\newtheorem{lemma}[thm]{Lemma}
\newtheorem{corol}[thm]{Corollary}
\newtheorem*{corol*}{Corollary}
\newtheorem{prop}[thm]{Proposition}
\newtheorem{defin}[thm]{Definition}
\newtheorem{Remark}[thm]{Remark}
\numberwithin{equation}{subsection}

\newtheorem{notas}[thm]{Notations}
\newtheorem{nota}[thm]{Notation}
\newtheorem{defis}[thm]{Definitions}
\newtheorem{hyp}[smfthm]{Hypothesis}
\newtheorem*{thm*}{Theorem}
\newtheorem*{prop*}{Proposition}
\newtheorem*{conj*}{Conjecture}

\def\Tm{{\mathbb T}}
\def\Um{{\mathbb U}}
\def\Am{{\mathbb A}}
\def\Fm{{\mathbb F}}
\def\Mm{{\mathbb M}}
\def\Nm{{\mathbb N}}
\def\Pm{{\mathbb P}}
\def\Qm{{\mathbb Q}}
\def\Zm{{\mathbb Z}}
\def\Dm{{\mathbb D}}
\def\Cm{{\mathbb C}}
\def\Rm{{\mathbb R}}
\def\Gm{{\mathbb G}}
\def\Lm{{\mathbb L}}
\def\Km{{\mathbb K}}
\def\Om{{\mathbb O}}
\def\Em{{\mathbb E}}
\def\Xm{{\mathbb X}}

\def\BC{{\mathcal B}}
\def\QC{{\mathcal Q}}
\def\TC{{\mathcal T}}
\def\ZC{{\mathcal Z}}
\def\AC{{\mathcal A}}
\def\CC{{\mathcal C}}
\def\DC{{\mathcal D}}
\def\EC{{\mathcal E}}
\def\FC{{\mathcal F}}
\def\GC{{\mathcal G}}
\def\HC{{\mathcal H}}
\def\IC{{\mathcal I}}
\def\JC{{\mathcal J}}
\def\KC{{\mathcal K}}
\def\LC{{\mathcal L}}
\def\MC{{\mathcal M}}
\def\NC{{\mathcal N}}
\def\OC{{\mathcal O}}
\def\PC{{\mathcal P}}
\def\UC{{\mathcal U}}
\def\VC{{\mathcal V}}
\def\XC{{\mathcal X}}
\def\SC{{\mathcal S}}
\def\RC{{\mathcal R}}

\def\BF{{\mathfrak B}}
\def\AF{{\mathfrak A}}
\def\GF{{\mathfrak G}}
\def\EF{{\mathfrak E}}
\def\CF{{\mathfrak C}}
\def\DF{{\mathfrak D}}
\def\JF{{\mathfrak J}}
\def\LF{{\mathfrak L}}
\def\MF{{\mathfrak M}}
\def\NF{{\mathfrak N}}
\def\XF{{\mathfrak X}}
\def\UF{{\mathfrak U}}
\def\KF{{\mathfrak K}}
\def\FF{{\mathfrak F}}

\def \longmapright#1{\smash{\mathop{\longrightarrow}\limits^{#1}}}
\def \mapright#1{\smash{\mathop{\rightarrow}\limits^{#1}}}
\def \lexp#1#2{\kern \scriptspace \vphantom{#2}^{#1}\kern-\scriptspace#2}
\def \linf#1#2{\kern \scriptspace \vphantom{#2}_{#1}\kern-\scriptspace#2}
\def \linexp#1#2#3 {\kern \scriptspace{#3}_{#1}^{#2} \kern-\scriptspace #3}

\def \lie {{\mathop{\mathrm{Lie}}\nolimits}}
\def \Ext{\mathop{\mathrm{Ext}}\nolimits}
\def \ad{\mathop{\mathrm{ad}}\nolimits}
\def \sh{\mathop{\mathrm{Sh}}\nolimits}
\def \irr{\mathop{\mathrm{Irr}}\nolimits}
\def \FH{\mathop{\mathrm{FH}}\nolimits}
\def \FPH{\mathop{\mathrm{FPH}}\nolimits}
\def \coh{\mathop{\mathrm{Coh}}\nolimits}
\def \res{\mathop{\mathrm{Res}}\nolimits}
\def \op{\mathop{\mathrm{op}}\nolimits}
\def \rec {\mathop{\mathrm{rec}}\nolimits}
\def \art{\mathop{\mathrm{Art}}\nolimits}
\def \vol {\mathop{\mathrm{vol}}\nolimits}
\def \cusp {\mathop{\mathrm{Cusp}}\nolimits}
\def \scusp {\mathop{\mathrm{Scusp}}\nolimits}
\def \Iw {\mathop{\mathrm{Iw}}\nolimits}
\def \JL {\mathop{\mathrm{JL}}\nolimits}
\def \speh {\mathop{\mathrm{Speh}}\nolimits}
\def \isom {\mathop{\mathrm{Isom}}\nolimits}
\def \Vect {\mathop{\mathrm{Vect}}\nolimits}
\def \groth {\mathop{\mathrm{Groth}}\nolimits}
\def \hom {\mathop{\mathrm{Hom}}\nolimits}
\def \deg {\mathop{\mathrm{deg}}\nolimits}
\def \val {\mathop{\mathrm{val}}\nolimits}
\def \det {\mathop{\mathrm{det}}\nolimits}
\def \rep {\mathop{\mathrm{Rep}}\nolimits}
\def \spec {\mathop{\mathrm{Spec}}\nolimits}
\def \fr {\mathop{\mathrm{Fr}}\nolimits}
\def \frob {\mathop{\mathrm{Frob}}\nolimits}
\def \ker {\mathop{\mathrm{Ker}}\nolimits}
\def \im {\mathop{\mathrm{Im}}\nolimits}
\def \Red {\mathop{\mathrm{Red}}\nolimits}
\def \red {\mathop{\mathrm{red}}\nolimits}
\def \aut {\mathop{\mathrm{Aut}}\nolimits}
\def \diag {\mathop{\mathrm{diag}}\nolimits}
\def \spf {\mathop{\mathrm{Spf}}\nolimits}
\def \Def {\mathop{\mathrm{Def}}\nolimits}
\def \nrd {\mathop{\mathrm{nrd}}\nolimits}
\def \supp {\mathop{\mathrm{Supp}}\nolimits}
\def \Id {{\mathop{\mathrm{Id}}\nolimits}}
\def \In {{\mathop{\mathrm{In}}\nolimits}}
\def \Ind{\mathop{\mathrm{Ind}}\nolimits}
\def \ind {\mathop{\mathrm{ind}}\nolimits}
\def \bad {\mathop{\mathrm{Bad}}\nolimits}
\def \top {\mathop{\mathrm{Top}}\nolimits}
\def \ker {\mathop{\mathrm{Ker}}\nolimits}
\def \coker {\mathop{\mathrm{Coker}}\nolimits}
\def \gal {{\mathop{\mathrm{Gal}}\nolimits}}
\def \Nr {{\mathop{\mathrm{Nr}}\nolimits}}
\def \rn {{\mathop{\mathrm{rn}}\nolimits}}
\def \tr {{\mathop{\mathrm{Tr~}}\nolimits}}
\def \Sp {{\mathop{\mathrm{Sp}}\nolimits}}
\def \st {{\mathop{\mathrm{St}}\nolimits}}
\def \sp{{\mathop{\mathrm{Sp}}\nolimits}}
\def \perv{\mathop{\mathrm{Perv}}\nolimits}
\def \tor {{\mathop{\mathrm{Tor}}\nolimits}}
\def \gr {{\mathop{\mathrm{gr}}\nolimits}}
\def \nilp {{\mathop{\mathrm{Nilp}}\nolimits}}
\def \obj {{\mathop{\mathrm{Obj}}\nolimits}}
\def \spl {{\mathop{\mathrm{Spl}}\nolimits}}
\def \unr {{\mathop{\mathrm{Unr}}\nolimits}}
\def \alg {{\mathop{\mathrm{Alg}}\nolimits}}
\def \grr {{\mathop{\mathrm{grr}}\nolimits}}
\def \cogr {{\mathop{\mathrm{cogr}}\nolimits}}
\def \coFil {{\mathop{\mathrm{coFil}}\nolimits}}

\def \rem{{\noindent\textit{Remark.~}}}
\def \rems{{\noindent\textit{Remarques:~}}}
\def \ext {{\mathop{\mathrm{Ext}}\nolimits}}
\def \End {{\mathop{\mathrm{End}}\nolimits}}

\def\semi{\mathrel{>\!\!\!\triangleleft}}
\let \DS=\displaystyle
\def\HT{{\mathop{\mathcal{HT}}\nolimits}}

\def \hi{\HC}
\newcommand*{\tarrow}{\relbar\joinrel\mid\joinrel\twoheadrightarrow}
\newcommand*{\harrow}{\lhook\joinrel\relbar\joinrel\mid\joinrel\rightarrow}
\newcommand*{\rarrow}{\relbar\joinrel\mid\joinrel\rightarrow}
\def \coim {{\mathop{\mathrm{Coim}}\nolimits}}
\def \can {{\mathop{\mathrm{can}}\nolimits}}
\def\LFF{{\mathscr L}}

\setcounter{secnumdepth}{3} \setcounter{tocdepth}{3}

\def \Fil{\mathop{\mathrm{Fil}}\nolimits}
\def \CoFil{\mathop{\mathrm{CoFil}}\nolimits}
\def \Fill{\mathop{\mathrm{Fill}}\nolimits}
\def \CoFill{\mathop{\mathrm{CoFill}}\nolimits}
\def\SF{{\mathfrak S}}
\def\PF{{\mathfrak P}}
\def \EFil{\mathop{\mathrm{EFil}}\nolimits}
\def \EFill{\mathop{\mathrm{EFill}}\nolimits}
\def \FP{\mathop{\mathrm{FP}}\nolimits}

\let \longto=\longrightarrow
\let \oo=\infty

\let \d=\delta
\let \k=\kappa

\renewcommand{\theequation}{\arabic{section}.\arabic{subsection}.\arabic{thm}}
\newcommand{\marque}{\addtocounter{thm}{1}
{\smallskip \noindent \textit{\thethm}~---~}}

\renewcommand\atop[2]{\ensuremath{\genfrac..{0pt}{1}{#1}{#2}}}

\newcommand\atopp[2]{\genfrac{}{}{0pt}{}{#1}{#2}}

\title[Galois irreducibility implies cohomology freeness]{Galois irreducibility implies 
cohomology freeness for KHT Shimura varieties}


\author{Boyer Pascal}
\email{boyer@math.univ-paris13.fr}
\address{Universit\'e Sorbonne Paris Nord \\
LAGA, CNRS, UMR 7539\\ 
F-93430, Villetaneuse (France) \\
CoLoSS: ANR-19-PRC}

\frontmatter

\begin{abstract}
Given a KHT Shimura variety with an action of its unramified Hecke algebra $\Tm$,
we proved in \cite{boyer-imj}, see also \cite{scholze-cara} for other PEL Shimura 
varieties, that its localized cohomology groups at a generic maximal ideal $\mathfrak m$ of 
$\Tm$, appear to be free.
In this work, we obtain the same result for $\mathfrak m$ such that its associated
galoisian $\overline \Fm_l$-representation $\overline{\rho_{\mathfrak m}}$ is irreducible,
under the hypothesis  that $[F(\exp(2i\pi/l):F]>d$ where $F$ is the reflex field, $d$
the dimension of the KHT Shimura variety and $l$ the residual characteristic.

\end{abstract}




\maketitle

\pagestyle{headings} \pagenumbering{arabic}

\tableofcontents
%
%

\section*{Introduction}
\renewcommand{\theequation}{\arabic{equation}}

\renewcommand{\thethm}{\Alph{thm}}

\backmatter

From Matsushima's formula and computations of $(\mathfrak G,K_\oo)$-cohomology,
we know that tempered automorphic representations contributions in the cohomology
of Shimura varieties with complex coefficients, is concentrated in middle degree.
If you consider cohomology with coefficients in a very regular local system, then
only tempered representations can contribute so that all of the cohomology is concentrated
in middle degree.

For $\overline \Zm_l$-coefficients and Shimura varieties of Kottwitz-Harris-Taylor types,
we proved in \cite{boyer-stabilization}, whatever is the weight of 
the coefficients, when the level is large enough at $l$, there are always non trivial torsion 
cohomology classes, so that the $\overline \Fm_l$-cohomology can not be concentrated in middle
degre. Thus if you want a $\overline \Fm_l$-analog of the previous $\overline \Qm_l$-statement, 
you must cut off some part of the cohomology. 

In \cite{boyer-imj} for KHT Shimura varieties, and more generally
in \cite{scholze-cara} for any PEL proper Shimura variety, we obtain such a result
under some genericness hypothesis which can be stated as follows.
Let $(\sh_K)_{K \subset G(\Am^\oo)}$ be a tower, indexed by open compact subgroups
$K$ of $G(\Am^\oo)$, of compact Shimura varieties of Kottwitz type associated to some 
similitude group $G$: we denote by $F=EF^+$ its reflex field
where $F^+$ is totally real and $E/\Qm$ is an imaginary quadratic extension.
Let then $\mathfrak m$ be a system of Hecke eigenvalues appearing in 
$H^{n_0}(\sh_K \times_F \overline F,\overline \Fm_l)$.
By the main result of \cite{scholze-torsion}, one can attach to such $\mathfrak m$,
a mod $l$ Galois representation 
$$\overline{\rho_{\mathfrak m}}: \gal(\overline F/F) \longrightarrow GL_d(\overline \Fm_l).$$
From \cite{scholze-cara} definition 1.9,
we say that $\mathfrak m$ is generic (resp. decomposed generic) at some split $p$ in $E$,  
if for all places $v$ of $F$ dividing $p$,
the set $\{ \lambda_1,\cdots,\lambda_n \}$ of eigenvalues of 
$\overline{\rho_{\mathfrak m}}(\frob_v)$  
satisfies $\lambda_i/\lambda_j \not \in \{ q_v^{\pm 1} \}$ for all $i \neq j$ 
(resp. and are pairwise distincts), where $q_v$ is the cardinal
of the residue field at $v$. Then under the hypothesis\footnote{In their new preprint, Caraiani and Scholze explained that, from an observation of Koshikawa, one can replace decomposed generic
by simply generic, in their main statement.} that there exists such $p$ with
$\mathfrak m$ generic at $p$, the integer $n_0$ above is necessary equals to
the relative dimension of $\sh_K$. In particular the 
$H^i(\sh_K \times_F \overline F,\overline \Zm_l)_{\mathfrak m}$
are all torsion free.
%
%
%

\medskip

In this work we consider the particular case of Kottwitz-Harris-Taylor Shimura varieties $\sh_K$ 
of \cite{h-t} associated to inner forms of $GL_d$. Exploiting the fact, which is particular to these Shimura
varieties, that the non supersingular Newton strata are geometrically induced, we are then able to
prove the following result which appears to be useful at least for our approach of
Ihara's lemma, cf. \cite{boyer-ihara}.

%

\smallskip

\noindent \textbf{\textit{Theorem}} --- \textit{
We suppose that $[F(\exp(2i\pi/l):F]>d$ 
Let $\mathfrak m$ be a system of Hecke eigenvalues 
such that $\overline{\rho_{\mathfrak m}}$
is irreducible, then the localized cohomology groups of $\sh_K$ with coefficients in any
$\overline \Zm_l$-local system $V_\xi$, are all free.}

Note that Koshikawa, cf. \cite{koshikawa}, starting from \cite{boyer-imj} and using techniques from group theory, proved a similar result in low dimension.

\rem Thanks to the hypothesis $[F(\exp(2i\pi/l):F]>d$ and by Cebotarev theorem,
we can pick places $v$ of $F$ such that the order $q_v$ of the residue field of $F$ at $v$,
is of order strictly greater than $d$  in $(\Zm/l\Zm)^\times$. 
This property is used at three places in the proof.
\begin{itemize}
\item For a place $v$ as above, there is no irreducible cuspidal representation
$\pi_v$ of $GL_g(F_v)$ with $g>1$ such that its modulo $l$ reduction has a supercuspidal
support made of characters, cf. the remark after \ref{nota-mvarrho}.
This simplification is completely harmless and if one wants to take care about
these cuspidal representations, it suffices to used the proposition 2.4.2 
of \cite{boyer-duke}.

\item With this hypothesis we also note that the pro-order of $GL_d(\OC_v)$ is invertible
modulo $l$ so that, concerning torsion cohomology classes, 
we can easily pass from infinite to maximal level at $v$, cf. for example the lemma \ref{lem-iIv}.

\item Finally in the last section, we are able to construct a
sequence of intervals contained in the set of eigenvalues of 
$\overline \rho_{\mathfrak m}(\frob_v)$ so that at the end we obtain a full set
$\{ \lambda q_v^n: n \in \Zm \}$ which is of order the order of $q_v$ modulo $l$ which
is trivially absurd if this order is strictly greater than the dimension $d$ of
$\overline \rho_{\mathfrak m}$.
\end{itemize}

%
%

%
%
%
%

The proof takes place in four main steps. 

(1) First we analyse the torsion in the
cohomology of Harris-Taylor perverse sheaves at some place
$v$ with infinite level at another place $w$. 
The main observation which rests on the symmetric role of both $v$ and $w$, 
cf. corollary \ref{coro-absurd}, 
is that, considering
$l$-torsion as a $\overline \Fm_l$-representation of $GL_d(F_w)$, then
we never encounter any irreducible generic sub-quotient with cuspidal
support made of characters.

(2) As one can compute the cohomology of the Shimura variety through the
spectral sequence associated to the exhaustive filtration of stratification
$\Fill^\bullet(\Psi_v)$ of the vanishing cycles perverse sheaf, where the
$E_1$ terms are given by the cohomology groups of the
Harris-Taylor perverse sheaves, we then observe that the various
lattices of $H^{d-1}_{free}(\sh_{K} \times_F \overline F,V_{\xi,\overline \Qm_l})_{\widetilde{\mathfrak m}}$ given by the integral cohomology, 
are only slightly modified from the ones given by the cohomology of the
Harris-Taylor perverses sheaves, in the sense that the $l$-torsion of the 
cokernel measuring the difference between two such lattices, as a representation
of $GL_d(F_w)$, does not have any 
irreducible generic sub-quotient with cuspidal
support made of characters, cf. proposition \ref{prop-lattice-psi}.

(3) However if the torsion sub-module of 
$H^{d-1}(\sh_{K} \times_F \overline F,V_{\xi,\overline \Zm_l})_{\mathfrak m}$
were not trivial, we prove, using the geometrical induced structure of the 
 Newton strata , that it must exists $\widetilde{\mathfrak m}$
such that the lattices of the previous point, are not isomorphic, cf. 
lemma \ref{lem-important}

(4) Finally in \S \ref{para-final}, 
for any $\widetilde{\mathfrak m} \subset \mathfrak m$, then 
$H^{d-1}(\sh_K \times_F \overline F,V_{\xi,\overline \Zm_l})_{\mathfrak m}$
induces a quotient stable lattice $\Gamma_{\widetilde{\mathfrak m}}$ of $(\Pi_{\widetilde{\mathfrak m}}^{\oo})^K \otimes
\rho_{\widetilde{\mathfrak m}}$. As
$\overline \rho_{\mathfrak m}$ is supposed to be irreducible, then this
lattice is isomorphic to a tensorial product of a stable lattice of 
$(\Pi_{\widetilde{\mathfrak m}}^{\oo})^K$ by a stable lattice of
$\rho_{\widetilde{\mathfrak m}}$. Then the idea is to start from the filtration
of the free quotient of $H^{d-1}(\sh_K \times_F \overline F,
V_{\xi,\overline \Zm_l})_{\mathfrak m}$ given
by the filtration of the nearby perverse sheaf, so that, using
diagrams as \ref{eq-extension}, we arrive at $\Gamma_{\widetilde{\mathfrak m}}$.
In the process we are able to construct an increasing sequence of interval
contained in the set of eigenvalues of $\overline \rho_{\mathfrak m}(\frob v)$
so that at the end we obtain a full set $\{ \lambda q_v^n : n \in \Zm \}$
which is of order the order of $q_v$ modulo $l$ which is, by hypothesis,
strictly greater than the dimension of $\overline \rho_{\mathfrak m}$, which is
absurd.

We refer the reader to the introduction of \S \ref{para-proof} for more details.

\mainmatter

\renewcommand{\theequation}{\arabic{section}.\arabic{subsection}.\arabic{thm}}

\renewcommand{\thethm}{\arabic{section}.\arabic{subsection}.\arabic{thm}}

\section{Recalls from \cite{boyer-imj}}

\subsection{Representations of $GL_d(K)$}
\label{para-gen}

We fix a finite extension $K/\Qm_p$ with residue field $\Fm_q$. We denote by $|-|$ its absolute
value.

For a representation $\pi$ of $GL_d(K)$ and $n \in \frac{1}{2} \Zm$, set 
$$\pi \{ n \}:= \pi \otimes q^{-n \val \circ \det}.$$

\begin{notas} \label{nota-ind}
For $\pi_1$ and $\pi_2$ representations of respectively $GL_{n_1}(K)$ and
$GL_{n_2}(K)$, we will denote by
$$\pi_1 \times \pi_2:=\ind_{P_{n_1,n_1+n_2}(K)}^{GL_{n_1+n_2}(K)}
\pi_1 \{ \frac{n_2}{2} \} \otimes \pi_2 \{-\frac{n_1}{2} \},$$
the normalized parabolic induced representation where for any sequence
$\underline r=(0< r_1 < r_2 < \cdots < r_k=d)$, we write $P_{\underline r}$ for 
the standard parabolic subgroup of $GL_d$ with Levi
$$GL_{r_1} \times GL_{r_2-r_1} \times \cdots \times GL_{r_k-r_{k-1}}.$$ 
\end{notas}

Recall that a representation
$\varrho$ of $GL_d(K)$ is called \emph{cuspidal} (resp. \emph{supercuspidal})
if it is not a subspace (resp. subquotient) of a proper parabolic induced representation.
When the field of coefficients is of characteristic zero then these two notions coincides,
but this is no more true for $\overline \Fm_l$.

\begin{defin} \label{defi-rep} (see \cite{zelevinski2} \S 9 and \cite{boyer-compositio} \S 1.4)
Let $g$ be a divisor of $d=sg$ and $\pi$ an irreducible cuspidal 
$\overline \Qm_l$-representation of $GL_g(K)$. 
The induced representation
$$\pi\{ \frac{1-s}{2} \} \times \pi \{ \frac{3-s}{2} \} \times \cdots \times \pi \{ \frac{s-1}{2} \}$$ 
holds a unique irreducible quotient (resp. subspace) denoted $\st_s(\pi)$ (resp.
$\speh_s(\pi)$); it is a generalized Steinberg (resp. Speh) representation.

Moreover the induced representation $\st_t(\pi \{ \frac{-r}{2} \}) \times \speh_r(\pi \{ \frac{t}{2} \} )$
(resp. of $\st_{t-1}(\pi \{ \frac{-r-1}{2} \}) \times \speh_{r+1}(\pi \{ \frac{t-1}{2} \} )$)
owns a unique irreducible subspace (resp. quotient), denoted
$LT_\pi(t-1,r)$.
\end{defin}

\rem These representations $LT_\pi(t-1,r)$ appear in the cohomology of the Lubin-Tate
spaces, cf. \cite{boyer-invent2}.

\begin{prop} \label{prop-red-modl} (cf. \cite{vigneras-livre} III.5.10)
Let $\pi$ be an irreducible cuspidal representation of $GL_g(K)$ with a stable
$\overline \Zm$-lattice\footnote{We say that $\pi$ is integral.}, then its modulo $l$ reduction
is irreducible and cuspidal but not necessary supercuspidal.
\end{prop}

The supercuspidal support of the modulo $l$ reduction of a cuspidal representation, 
is a segment associated to some irreducible $\overline \Fm_l$-supercuspidal representation
$\varrho$ of $GL_{g_{-1}(\varrho)}(F_v)$ with $g=g_{-1}(\varrho) t$ where $t$ is 
either equal to $1$ or of the following shape
$t=m(\varrho)l^u$ with $u \geq 0$ and where $m(\varrho)$ is defined as follows.

\begin{nota} \label{nota-mvarrho}
We denote by $m(\varrho)$ the order of the Zelevinsky line 
$\{ \varrho(\delta): \delta \in \Zm \}$ of
$\varrho$ if it is not equal to $1$, otherwise $m(\varrho)=l$.
\end{nota}

\rem When $\varrho$ is the trivial representation then $m(1_v)$ is either the order of $q$ modulo $l$
when it is $>1$, otherwise $m(1_v)=l$. 
We say that such $\pi_v$ is of $\varrho$-type $u$ with $u \geq -1$.

\begin{nota}
For $\varrho$ an irreducible $\overline \Fm_l$-supercuspidal representation, we denote by
$\cusp_\varrho$ (resp. $\cusp_\varrho(u)$ for some $u \geq -1$) 
the set of equivalence classes of irreducible
$\overline \Qm_l$-cuspidal representations whose modulo $l$ reduction has 
for supercuspidal
support a segment associated to $\varrho$ (resp. of $\varrho$-type $u$).
\end{nota}

Let $u \geq 0$, $\pi_{v,u} \in \cusp_{\varrho}(u)$ and $\tau=\pi_{v,u}[s]_{D}$. Let then denote
by $\iota$ the image of $\speh_s(\varrho)$ by the modulo $l$ Jacquet-Langlands 
correspondence
defined at \S 1.2.4 de \cite{dat-jl}. Then the modulo $l$ reduction of $\tau$ is isomorphic to
\addtocounter{thm}{1}
\begin{equation} \label{eq-red-tau}
\iota \{-\frac{m(\tau)-1}{2} \} \oplus \iota \{-\frac{m(\tau)-3}{2} \} \oplus \cdots \oplus \iota \{ \frac{m(\tau)-1}{2} \}
\end{equation}
where $\iota \{ n \}:=\iota \otimes q^{-n \val \circ \nrd}$.

We want now to recall the notion of level of non degeneracy from \cite{zelevinski1} \S 4. 
The mirabolic subgroup
$M_d(K)$ of $GL_d(K)$ is the sub-group of matrices with last row $(0,\cdots,0,1)$: we denote by
$$V_d(K)=\{ (m_{i,j} \in P_d(K):~m_{i,j}= \delta_{i,j} \hbox{ for } j < n \}.$$ 
its unipotent radical. We fix a non trivial character $\psi$ of $K$ and let $\theta$ be the
character of $V_d(K)$ defined by $\theta( (m_{i,j}))=\psi(m_{d-1,d})$.
For $G=GL_r(K)$ or $M_r(K)$, we denote by $\alg(G)$ the abelian category of algebraic 
representations of $G$ and, following \cite{zelevinski1}, we introduce
$$\Psi^-: \alg(M_d(K)) \longrightarrow \alg(GL_{d-1}(K), \qquad \Phi^-: \alg (M_d) \longrightarrow
\alg (M_{d-1}(K))$$
defined by $\Psi^-=r_{V_d,1}$ (resp. $\Phi^-=r_{V_d,\theta}$) the functor of $V_{d-1}$
coinvariants (resp. $(V_{d-1},\theta)$-coinvariants), cf. \cite{zelevinski1} 1.8.
We also introduce the normalized compact induced functor
$$\Psi^+:=i_{V,1}: \alg(GL_{d-1}(K)) \longrightarrow \alg (M_d(K)),$$ 
$$\Phi^+:=i_{V,\theta}: \alg(M_{d-1}(K)) \longrightarrow \alg(M_d(K)).$$

\begin{prop} (\cite{zelevinski1} p451)
\begin{itemize}
\item The functors $\Psi^-$, $\Psi^+$, $\Phi^-$ and $\Phi^+$ are exact.

\item $\Phi^- \circ \Psi^+=\Psi^- \circ \Phi^+=0$.

\item $\Psi^-$ (resp. $\Phi^+$) is left adjoint to $\Psi^+$ (resp. $\Phi^-$) and the
following adjunction maps 
$$\Id \longrightarrow \Phi^- \Phi^+, \qquad \Psi^+ \Psi^- \longrightarrow \Id,$$
are isomorphisms meanwhile
$$0 \rightarrow \Phi^+ \Phi^- \longrightarrow \Id \longrightarrow \Psi^+ \Psi^- 
\rightarrow 0.$$
\end{itemize}
\end{prop}

\begin{defin}
For $\tau \in \alg(M_d(K))$, the representation 
$$\tau^{(k)}:=\Psi^- \circ (\Phi^-)^{k-1}(\tau)$$
is called the $k$-th derivative of $\tau$. If $\tau^{(k)}\neq 0$ and $\tau^{(m)}=0$
for all $m > k$, then $\tau^{(k)}$ is called the highest derivative of $\tau$.
\end{defin}

\begin{nota} (cf. \cite{zelevinski2} 4.3) \label{nota-nondegeneracy}
Let $\pi \in \alg(GL_d(K))$ (or $\pi \in \alg(M_d(K)$). The maximal number $k$ such that 
$(\pi_{|M_d(K)})^{(k)} \neq (0)$ is called the level of non-degeneracy of $\pi$ and 
denoted by $\lambda(\pi)$. We can also iterate the construction so that at the end we obtain 
a partition $\underline{\lambda(\pi)}$ of $d$.
\end{nota}

\begin{defin}
A representation $\pi$ of $GL_d(K)$, over $\overline \Qm_l$ or $\overline \Fm_l$, is then said generic
if its level of non degeneracy $\lambda(\pi)$ is equal to $d$.
\end{defin}

\rem Over $\overline \Qm_l$, an irreducible generic representation of
$GL_d(K)$ looks like $\st_{t_1}(\pi_1) \times \cdots \times \st_{t_r}(\pi_r)$
where $\pi_1,\cdots, \pi_r$ are irreducible cuspidal representations. Note
moreover that the modulo $l$ reduction of any irreducible generic representation
owns a unique generic irreducible sub-quotient.

\subsection{Shimura varieties of KHT type}

\label{para-geo}

Let $F=F^+ E$ be a CM field where $E/\Qm$ is quadratic imaginary and 
$F^+/\Qm$ is
totally real with a fixed real embedding $\tau:F^+ \hookrightarrow \Rm$. For a place $v$ of $F$,
we will denote by
\begin{itemize}
\item $F_v$ the completion of $F$ at $v$,

\item $\OC_v$ the ring of integers of $F_v$,

\item $\varpi_v$ a uniformizer,

\item $q_v$ the cardinal of the residual field $\kappa(v)=\OC_v/(\varpi_v)$.
\end{itemize}
Let $B$ be a division algebra with center $F$, of dimension $d^2$ such that at every place $x$ of $F$,
either $B_x$ is split or a local division algebra and suppose $B$ provided with an involution of
second kind $*$ such that $*_{|F}$ is the complex conjugation. For any
$\beta \in B^{*=-1}$, denote by $\sharp_\beta$ the involution $x \mapsto x^{\sharp_\beta}=\beta x^*
\beta^{-1}$ and let $G/\Qm$ be the group of similitudes, denoted by 
$G_\tau$ in \cite{h-t}, defined for every $\Qm$-algebra $R$ by 
$$
G(R)  \simeq   \{ (\lambda,g) \in R^\times \times (B^{op} \otimes_\Qm R)^\times  \hbox{ such that } 
gg^{\sharp_\beta}=\lambda \}
$$
with $B^{op}=B \otimes_{F,c} F$. 
If $x$ is a place of $\Qm$ split $x=yy^c$ in $E$ then 
\addtocounter{thm}{1}
\begin{equation} \label{eq-facteur-v}
G(\Qm_x) \simeq (B_y^{op})^\times \times \Qm_x^\times \simeq \Qm_x^\times \times
\prod_{z_i} (B_{z_i}^{op})^\times,
\end{equation}
where, identifying places of $F^+$ over $x$ with places of $F$ over $y$,
$x=\prod_i z_i$ in $F^+$.

\noindent \textbf{Convention}: for $x=yy^c$ a place of $\Qm$ split in $E$ and $z$ 
a place of $F$ over $y$, we shall make throughout the text, the following abuse of notation by denoting 
$G(F_z)$ in place of the factor $(B_z^{op})^\times$ in the formula (\ref{eq-facteur-v}).

In \cite{h-t}, the authors justify the existence of some $G$ like before such that moreover
\begin{itemize}
\item if $x$ is a place of $\Qm$ non split in $E$ then $G(\Qm_x)$ is quasi split;

\item the invariants of $G(\Rm)$ are $(1,d-1)$ for the embedding $\tau$ and $(0,d)$ for the others.
\end{itemize}

As in  \cite{h-t} bottom of page 90, a compact open subgroup $U$ of $G(\Am^\oo)$ is said 
\emph{small enough}
if there exists a place $x$ such that the projection from $U^v$ to $G(\Qm_x)$ does not contain any 
element of finite order except identity.

\begin{nota}
Denote by $\IC$ the set of open compact subgroups small enough of $G(\Am^\oo)$.
For $I \in \IC$, write $\sh_{I,\eta} \longrightarrow \spec F$ for the associated
Shimura variety of Kottwitz-Harris-Taylor type.
\end{nota}

\begin{defin} \label{defi-spl}
Denote by $\spl$ the set of  places $v$ of $F$ such that $p_v:=v_{|\Qm} \neq l$ is split in $E$ and let
$B_v^\times \simeq GL_d(F_v)$.  For each $I \in \IC$, we write
$\spl(I)$ for the subset of $\spl$ of places which does not divide $I$.
\end{defin}

In the sequel, $v$ and $w$ will denote places of $F$ in $\spl$. 
For such a place $v$,
the scheme $\sh_{I,\eta}$ has a projective model $\sh_{I,v}$ over $\spec \OC_v$
with special fiber $\sh_{I,s_v}$. For $I$ going through $\IC$, the projective system $(\sh_{I,v})_{I\in \IC}$ 
is naturally equipped with an action of $G(\Am^\oo) \times \Zm$ such that any
$w_v$ in the Weil group $W_v$ of $F_v$ acts by $-\deg (w_v) \in \Zm$,
where $\deg=\val \circ \art^{-1}$ and $\art^{-1}:W_v^{ab} \simeq F_v^\times$ is the isomorphism of Artin
sending the geometric Frobenius to uniformizers.

\begin{notas} 
For $I \in \IC$, the Newton stratification of the geometric special fiber $\sh_{I,\bar s_v}$ is denoted by
$$\sh_{I,\bar s_v}=:\sh^{\geq 1}_{I,\bar s_v} \supset \sh^{\geq 2}_{I,\bar s_v} \supset \cdots \supset 
\sh^{\geq d}_{I,\bar s_v}$$
where $\sh^{=h}_{I,\bar s_v}:=\sh^{\geq h}_{I,\bar s_v} - \sh^{\geq h+1}_{I,\bar s_v}$ is an affine 
scheme, smooth of pure dimension $d-h$ built up by the geometric 
points whose connected part of its Barsotti-Tate group is of rank $h$.
For each $1 \leq h <d$, write
$$i_{h}:\sh^{\geq h}_{I,\bar s_v} \hookrightarrow \sh^{\geq 1}_{I,\bar s_v}, \quad
j^{\geq h}: \sh^{=h}_{I,\bar s_v} \hookrightarrow \sh^{\geq h}_{I,\bar s_v},$$
and $j^{=h}=i_h \circ j^{\geq h}$.
\end{notas}

Let $\sigma_0:E \hookrightarrow
\overline{\Qm}_l$ be a fixed embedding and write $\Phi$ for the set of embeddings 
$\sigma:F \hookrightarrow \overline \Qm_l$ whose restriction to $E$ equals $\sigma_0$.
There exists then, cf. \cite{h-t} p.97, an explicit bijection between irreducible algebraic representations 
$\xi$ of $G$ over $\overline \Qm_l$ and $(d+1)$-uple
$\bigl ( a_0, (\overrightarrow{a_\sigma})_{\sigma \in \Phi} \bigr )$
where $a_0 \in \Zm$ and for all $\sigma \in \Phi$, we have $\overrightarrow{a_\sigma}=
(a_{\sigma,1} \leq \cdots \leq a_{\sigma,d} )$. 
We then denote by 
$$V_{\xi,\overline \Zm_l}$$ 
the associated $\overline \Zm_l$-local system on $\sh_\IC$.
Recall that an irreducible automorphic representation $\Pi$ is said $\xi$-cohomological if there exists
an integer $i$ such that
$$H^i \bigl ( ( \lie ~G(\Rm)) \otimes_\Rm \Cm,U,\Pi_\oo \otimes \xi^\vee \bigr ) \neq (0),$$
where $U$ is a maximal open compact subgroup modulo the center of $G(\Rm)$.
Let $d_\xi^i(\Pi_\oo)$ be the dimension of this cohomology group.

\subsection{Cohomology of the Newton strata}
\label{para-hecke}

\begin{nota} \label{nota-hixi}
For $1 \leq h \leq d$, let $\IC_v(h)$ be the set of open compact subgroups 
$$U_v(\underline m,h):= 
U_v(\underline m^v) \times \left ( \begin{array}{cc} I_h & 0 \\ 0 & K_v(m_1) \end{array} \right ),$$
where $K_v(m_1)=\ker \bigl ( GL_{d-h}(\OC_v) \longrightarrow GL_{d-h}(\OC_v/ (\varpi_v^{m_1})) \bigr )$.
We then denote by $[H^i(h,\xi)]$ (resp. $[H^i_!(h,\xi)]$) the image of
$$\lim_{\atop{\longrightarrow}{I \in \IC_v(h)}} H^i(\sh_{I,\bar s_v,1}^{\geq h}, V_{\xi,\overline \Qm_l}[d-h]) 
\qquad \hbox{resp. }
\lim_{\atop{\longrightarrow}{I \in \IC_v(h)}} H^i(\sh_{I,\bar s_v,1}^{\geq h}, j^{\geq h}_{1,!} V_{\xi,\overline \Qm_l}[d-h]) 
$$ 
inside the Grothendieck $\groth(v,h)$ of admissible representations of 
$G(\Am^{\oo}) \times GL_{d-h}(F_v) \times \Zm$.
\end{nota}

\rem An element $\sigma \in W_v$ acts through $-\deg \sigma \in \Zm$ and $\Pi_{p_v,0}(\art^{-1} (\sigma))$.
We moreover consider the action of $GL_{h}(F_v)$ through  
$\val \circ \det: GL_{h}(F_v) \longrightarrow \Zm$ and finally
$P_{h,d}(F_v)$ through its Levi factor $GL_{h}(F_v) \times GL_{d-h}(F_v)$, i.e.
its unipotent radical acts trivially.

From \cite{boyer-imj} proposition 3.6, for any irreducible tempered automorphic representation 
$\Pi$ of $G(\Am)$ and for every $i \neq 0$, the $\Pi^{\oo,v}$-isotypic component of
$[H^i(h,\xi)]$ and $[H^i_!(h,\xi)]$ are zero. About the case $i=0$, for $\Pi$ an irreducible
automorphic tempered representation $\xi$-cohomological, its local component at $v$ is generic and so looks like
$$\Pi_v \simeq \st_{t_1}(\pi_{v,1}) \times \cdots \times \st_{t_u}(\pi_{v,u}),$$
where for $i=1,\cdots,u$, $\pi_{v,i}$ is an irreducible cuspidal representation of est une 
$GL_{g_i}(F_v)$.

\begin{prop} \label{prop-temperee-explicite} (cf. \cite{boyer-imj} proposition 3.9)
With the previous notations, we order the $\pi_{v,i}$ such that the first $r$-ones are unramified
characters. Then the $\Pi^{\oo,v}$-isotypic component of $[H^0(h,\xi)]$ is then equals to
$$\Bigl ( \frac{\sharp \ker^1(\Qm,G)}{d} \sum_{\Pi' \in \UC_G(\Pi^{\oo,v})} 
m(\Pi') d_\xi(\Pi'_\oo) \Bigr )  \Bigl ( \sum_{1 \leq k \leq r:~ t_k=h} \Pi_v^{(k)} \otimes \chi_{v,k}
\chi^{\frac{d-h}{2}} \Bigr )$$
where
\begin{itemize}
\item $ \ker^1(\Qm,G)$ is the subset of elements of $H^1(\Qm,G)$ which become trivial
in $H^1(\Qm_{p'},G)$ for every prime $p'$;

\item $\Pi_v^{(k)}:= \st_{t_1}(\chi_{v,1}) \times \cdots \times \st_{t_{k-1}}(\chi_{v,k-1}) \times
\st_{t_{k+1}}(\chi_{v,k+1}) \times \cdots \times \st_{t_u}(\chi_{v,u})$ and
 
\item $\Xi:\frac{1}{2} \Zm \longrightarrow \overline \Zm_l^\times$ 
is defined by $\Xi(\frac{1}{2})=q_v^{\frac{1}{2}}$. 

\item $\UC_G(\Pi^{\oo,v})$ is the set of equivalence classes of irreducible automorphic 
representations $\Pi'$ of $G(\Am)$ such that $(\Pi')^{\oo,v} \simeq \Pi^{\oo,v}$.
\end{itemize}
\end{prop}

\rem In particular if $[H^0(h,\xi)]$ has non trivial invariant vectors under some open
compact subgroup $I \in \IC_v(h)$ which is maximal at $v$, then the local component
of $\Pi$ at $v$ is of the following shape $\st_h(\chi_{v,1}) \times \chi_{v,2} \times \cdots \chi_{v,d-h}$
where the $\chi_{v,i}$ are unramified characters.

\begin{defin}
For a finite set $S$ of places of $\Qm$ containing the places where $G$ is ramified, 
denote by $\Tm^S_{abs}:=\prod_{x \not \in S} \Tm_{x,abs}$ the 
abstract unramified  Hecke algebra
where
$\Tm_{x,abs} \simeq \overline \Zm_l[X^{un}(T_x)]^{W_x}$ for $T_x$ a split torus,
$W_x$ the spherical Weyl group and $X^{un}(T_x)$ the set of $\overline \Zm_l$-unramified 
characters of $T_x$. 
\end{defin}

\noindent \textit{Example}.
For $w \in \spl$, we have
$$\Tm_{w,abs}=\overline \Zm_l \bigl [T_{w,i}:~ i=1,\cdots,d \bigr ],$$
where $T_{w,i}$ is the characteristic function of
$$GL_d(\OC_w) \diag(\overbrace{\varpi_w,\cdots,\varpi_w}^{i}, \overbrace{1,\cdots,1}^{d-i} ) 
GL_d(\OC_w) \subset  GL_d(F_w).$$

\begin{nota} 
Let $\Tm^S_\xi$ be the image of $\Tm^S_{abs}$ inside 
$$\bigoplus_{i =0}^{2d-2}
\lim_{\atop{\rightarrow}{I}} H^i(\sh_{I,\bar \eta},V_{\xi,\overline \Qm_l})$$
where the limit is taken over 
the ideals $I$ which are maximal at each places outside $S$. For $I$
an open compact subgroup maximal at each places outside $S$,
we will also denote by $\Tm^S_{I,\xi}$ the image of $\Tm^S_{abs}$ inside
$H^{d-1}(\sh_{I,\bar \eta},V_{\xi,\overline \Qm_l})$.
\end{nota}

Let state some remarks about these Hecke algebras.
\begin{itemize}
\item In \cite{boyer-imj}, we proved that if $\mathfrak m$ is a maximal ideal of 
$\Tm^S_\xi$ such that there exists $i$ with 
$H^i(\sh_{I,\bar \eta},V_{\xi,\overline \Zm_l})_{\mathfrak m}
\neq (0)$ for $I$ maximal at each places outside $S$, then
$(\Tm^S_\xi)_{\mathfrak m} \neq (0)$, i.e. torsion cohomology classes raise in 
characteristic zero. In particular to define $\Tm^S_{I,\xi}$ there is no difference
taking cohomology with $\overline \Zm_l$ or $\overline \Qm_l$ coefficients,
and consider torsion classes or only the free quotients.

\item As explained in the introduction, we will consider maximal ideals
$\mathfrak m$ such that $\overline \rho_{\mathfrak m}$ is irreducible so that
the $\overline Qm_l$-cohomology groups are all concentrated in middle degree,
i.e. in degree $0$ if we deal with perverse sheaves.

\item With the notations of \S \ref{para-HT} about the Harris-Taylor local systems,
in \cite{boyer-compositio}, we proved that, except for the $GL_d(F_v)$-action,
the irreducible sub-quotients of the
$\overline \Qm_l$-cohomology groups of $j^{\geq tg}_! HT(\pi_v,\Pi_t)$ or 
$\lexp p j^{\geq tg}_{!*} HT(\pi_v,\Pi_t)$ are also sub-quotients of
the cohomology of $\sh_{I,\bar \eta}$. Moreover their torsion classes also
raises in characteristic zero. In particular if $\mathfrak m$ is such that
$\overline \rho_{\mathfrak m}$ is irreducible, then the image of $\Tm^S_\xi$
inside $H^0(\sh_{I,\bar s_v},\lexp p j^{\geq tg}_{!*} HT(\pi_v,\Pi_t))_{\mathfrak m}$
factors through $\Tm^S_{I',\chi}$ for $I'$ such that $(I')^v=I^v$. The same
is also true for $j^{\geq tg}_{!} HT(\pi_v,\Pi_t)$. In the case where $\pi_v$
is a unramified character, then you can take $I'_v$ for the Iwahori subgroup
of $GL_d(\OC_v)$.

\end{itemize}

The minimal prime ideals of $\Tm^S_{\xi}$ are the prime ideals above the zero ideal of 
$\overline \Zm_l$ and are then in bijection with the prime ideals of 
$\Tm^S_{\xi} \otimes_{\overline \Zm_l} \overline \Qm_l$. 
To such an ideal, which corresponds
to give a collection of Satake parameters, is then associated a unique near equivalence class in
the sense of \cite{y-t}, denoted by $\Pi_{\widetilde{\mathfrak m}}$, which is the finite set of
irreducible automorphic cohomological representations whose multi-set of Satake parameters
at each place $x \in \unr(I)$, is given by $S_{\widetilde{\mathfrak m}}(x)$ the multi-set of roots of
the Hecke polynomial
$$P_{\widetilde{\mathfrak{m}},w}(X):=\sum_{i=0}^d(-1)^i q_w^{\frac{i(i-1)}{2}} 
T_{w,i,\widetilde{\mathfrak m}} X^{d-i} \in \overline \Qm_l[X]$$
i.e.
$$S_{\widetilde{\mathfrak{m}}}(w) := \bigl \{ \lambda \in \Tm^S_{\xi} \otimes_{\overline \Zm_l}
\overline \Qm_l / \widetilde{\mathfrak m} \simeq \overline \Qm_l \hbox{ such that }
P_{\widetilde{\mathfrak{m}},w}(\lambda)=0 \bigr \}.$$
Thanks to \cite{h-t} and \cite{y-t}, we denote by
$$\rho_{\widetilde{\mathfrak m}}:\gal(\overline F/F) \longrightarrow GL_d(\overline \Qm_l)$$ 
the Galois representation associated to any $\Pi \in \Pi_{\widetilde{\mathfrak m}}$.
Recall that the modulo $l$ reduction of $\rho_{\widetilde{\mathfrak m}}$ depends only
of $\mathfrak m$, and was denoted above $\overline{\rho_{\mathfrak m}}$.
For every $w \in \spl(I)$, we also denote by $S_{\mathfrak{m}}(w)$ the multi-set of modulo $l$
Satake parameters at $w$ given as the multi-set of roots of
$$P_{\mathfrak{m},w}(X):=\sum_{i=0}^d(-1)^i q_w^{\frac{i(i-1)}{2}} \overline{T_{w,i}} X^{d-i} 
\in \overline \Fm_l[X]$$
i.e.
$$S_{\mathfrak{m}}(w) := \bigl \{ \lambda \in \Tm^S_\xi/\mathfrak m \simeq \overline \Fm_l 
\hbox{ such that } P_{\mathfrak{m},w}(\lambda)=0 \bigr \}.$$
Using the arguments of \cite{cara-m} and following \cite{scholze-torsion} V.4.4, 
we then deduce the existence, for $\gal_{F,S}$ the Galois group of the maximal
extension of $F$ unramified outside $S$, of
\addtocounter{thm}{1}
\begin{equation} \label{eq-rhom}
\rho_{\xi,\mathfrak m}: \gal_{F,S} \longrightarrow GL_d((\Tm^S_\xi)_{\mathfrak m})
\end{equation}
interpolating the $\rho_{\widetilde{\mathfrak m}}$, so that in particular for all $u \not \in S$,
$\det(1-X \frob_u|\rho_{\mathfrak m})$ is equal to the 
Hecke polynomial.

\rem In  \cite{scholze-torsion}, the author constructs 
$\rho_{\mathfrak m}: \gal_{F,S} \longrightarrow GL_d((\Tm^S_\xi)_{\mathfrak m}/J)$
where $J$ is a nilpotent ideal but it seems from incoming work that on can arrange
$J$ to be zero.

%
%
%
%

\section{About the nearby cycle perverse sheaf}

Our strategy to compute the cohomology of the KHT-Shimura variety $\sh_{I,\bar \eta}$ with
coefficients in $V_{\xi,\overline \Zm_l}$, is to realize it as the outcome of the nearby cycles 
spectral sequence at some place $v \in \spl$.

Note that the role of the local system $V_{\xi,\overline \Zm_l}$
associated to $\xi$ is completely harmless when
dealing with sheaves: one just have to add a tensor product with it to all the
statements without the index $\xi$. In the following we will sometimes not mention the index
$\xi$ in the statements to make formulas more readable. Of course when looking at the 
cohomology groups, the role of $V_{\xi,\overline \Zm_l}$ is crucial as it selects the
automorphic representations which contribute to the cohomology.

\subsection{The case where the level at $v$ is maximal}
\label{para-max}

By the smooth base change theorem, we have
$H^i(\sh_{I,\bar \eta_v}, V_\xi) \simeq H^i(\sh_{I,\bar s_v},V_\xi)$.
As for each $1 \leq h \leq d-1$, the open Newton stratum $\sh^{=h}_{I,\bar s_v}$ is affine
then $H^i_c(\sh^{=h}_{I,\bar s_v},V_{\xi,\overline \Zm_l}[d-h])$ is zero for $i<0$ and free for $i=0$.
Using this property and the following short exact sequence
of free perverse sheaves
\begin{multline*}
0 \rightarrow i_{h+1,*} V_{\xi,\overline \Zm_l,|\sh_{I,\bar s_v}^{\geq h+1}}[d-h-1] \longrightarrow 
j^{\geq h}_! j^{\geq h,*} V_{\xi,\overline \Zm_l,|\sh^{\geq h}_{I,\bar s_v}}[d-h] \\ \longrightarrow 
V_{\xi,\overline \Zm_l,|\sh^{\geq h}_{I,\bar s_v}}[d-h] \rightarrow 0,
\end{multline*}
we then obtain for every $i>0$
\addtocounter{thm}{1}
\begin{equation} \label{eq-sec}
0 \rightarrow H^{-i-1}(\sh^{\geq h}_{I,\bar s_v},V_{\xi,\overline \Zm_l}[d-h]) \longrightarrow
H^{-i}(\sh^{\geq h+1}_{I,\bar s_v},V_{\xi,\overline \Zm_l}[d-h-1]) \rightarrow 0,
\end{equation}
and for $i=0$,
\addtocounter{thm}{1}
\begin{multline} \label{eq-sec0}
0 \rightarrow H^{-1}(\sh^{\geq h}_{I,\bar s_v},V_{\xi,\overline \Zm_l}[d-h]) \longrightarrow
H^{0}(\sh^{\geq h+1}_{I,\bar s_v},V_{\xi,\overline \Zm_l}[d-h-1]) \longrightarrow \\
H^0(\sh^{\geq h}_{I,\bar s_v},j^{\geq h}_! j^{\geq h,*} V_{\xi,\overline \Zm_l}[d-h] ) \longrightarrow
H^{0}(\sh^{\geq h}_{I,\bar s_v},V_{\xi,\overline \Zm_l}[d-h]) \rightarrow \cdots
\end{multline}
In \cite{boyer-imj}, arguing by induction from $h=d$ to $h=1$, we prove that for a maximal
ideal $\mathfrak m$ of $\Tm^S_{\xi}$ such that $S_{\mathfrak m}(v)$ does not contain any subset 
of the form $\{ \alpha,q_v \alpha \}$, all the cohomology groups 
$H^i(\sh_{I,\bar s_v}^{\geq h},V_{\xi,\overline \Zm_l})_{\mathfrak m}$ are free: note that in order
to deal with $i\geq 0$, one has to use the Grothendieck-Verdier duality.

Without this hypothesis, arguing similarly, we conclude that any torsion cohomology class
comes from a non strict map 
\addtocounter{thm}{1}
\begin{equation} \label{eq-map-strict-fini}
H^{0}_{free}(\sh^{\geq h+1}_{I,\bar s_v},V_{\xi,\overline \Zm_l}[d-h-1])_{\mathfrak m} \longrightarrow
H^0(\sh^{\geq h}_{I,\bar s_v},j^{\geq h}_! j^{\geq h,*} V_{\xi,\overline \Zm_l}[d-h] )_{\mathfrak m}.
\end{equation}
In particular it raises in characteristic zero to some free subquotient of
$H^0(\sh^{\geq h}_{I,\bar s_v},j^{\geq h}_! j^{\geq h,*} V_{\xi,\overline \Zm_l}[d-h] )_{\mathfrak m}$.

%

We argue by absurdity and we suppose there exists $I \in \IC$ such that there
exists non trivial torsion cohomology classes in the $\mathfrak m$-localized cohomology of 
$\sh_{I,\bar \eta_v}$ with coefficients in $V_{\xi,\overline \Zm_l}$. Fix such finite level $I$.

\begin{prop} \label{prop-h0I} (cf. \cite{boyer-imj} lemme 4.13)
Consider $h_0(I)$ maximal such that there exists $i \in \Zm$ with
$H^{d-h_0(I)+i}(\sh_{I,\bar s_v}^{\geq h_0(I)}, V_{\xi,\overline \Zm_l})_{\mathfrak m,tor} 
\neq (0)$. Then we have the following properties:
\begin{itemize}
\item $i=0,1$;

\item for all $1 \leq h \leq h_0(I)$ and $i<h-h_0(I)$,
$$H^{d-h+i}(\sh_{I,\bar s_v}^{\geq h} V_{\xi,\overline \Zm_l})_{\mathfrak m,tor} =(0)$$ 
while for $i=h-h_0(I)$ it is non trivial.
\end{itemize}
\end{prop}

\rem Note that any system of Hecke eigenvalues
$\mathfrak m$ of $\Tm^S_\xi$ inside the torsion of some 
$H^i(\sh_{I,\bar \eta_v},V_{\xi,\overline \Zm_l})$ raises in characteristic zero, i.e. 
is associated to a minimal prime ideal $\widetilde{\mathfrak m}$ of $\Tm^{S \cup \{ v \}}_{\xi}$. 
More precisely, using the remark following the proposition \ref{prop-temperee-explicite},
there exists $\widetilde{\mathfrak m} \subset \mathfrak m$ such that
the local component at $v$ of $\pi_{\widetilde{\mathfrak m}}$ is isomorphic to
$\st_{h_0(I)+1}(\chi_v) \times \chi_{v,1} \times \cdots \times \chi_{v,d-h_0(I)-1}$
where $\chi_v, \chi_{v,1},\cdots , \chi_{v,d-h_0(I)-1}$ are characters of $F_v^\times$.

\subsection{Harris-Taylor perverse sheaves over $\overline \Zm_l$}
\label{para-HT}

Consider now the ideals $I^v(n):=I^vK_v(n)$ where
$K_v(n):=\ker(GL_d(\OC_v) \twoheadrightarrow GL_d(\OC_v/\MC_v^n))$.
Recall then that $\sh_{I^v(n),\bar s_v}^{=h}$ is geometrically induced
under the action of the parabolic subgroup $P_{h,d}(\OC_v/\MC_v^n)$, defined as the
stabilizer of the first $h$ vectors of the canonical basis of $F_v^d$. Concretely this means there
exists a closed subscheme $\sh_{I^v(n),\bar s_v,\overline{1_{h}}}^{=h}$ stabilized by the Hecke 
action of $P_{h,d}(F_v)$ and such that
$$\sh_{I^v(n),\bar s_v}^{=h} = \sh_{I^v(n),\bar s_v,\overline{1_{h}}}^{=h} 
\times_{P_{h,d}(\OC_v/\MC_v^n)} GL_d(\OC_v/\MC_v^n),$$
meaning that $\sh_{I^v(n),\bar s_v}^{=h} $ is the disjoint union of copies of
$\sh_{I^v(n),\bar s_v,\overline{1_{h}}}^{=h}$ indexed by 
$GL_d(\OC_v/\MC_v^n)/P_{h,d}(\OC_v/\MC_v^n)$ and 
exchanged by the action of
$GL_d(\OC_v/\MC_v^n)$.

\begin{nota} \label{nota-strata}
For any $g \in GL_d(\OC_v/\MC_v^n)/P_{h,d}(\OC_v/\MC_v^n)$, we denote by
$\sh^{=h}_{I^v(n),\bar s_v,g}$ the \emph{pure} Newton stratum defined as the image of
$\sh_{I^v(n),\bar s_v,\overline{1_{h}}}^{=h}$ by $g$.
Its closure in $\sh_{I^v(n),\bar s_v}$ is then denoted by
$\sh^{\geq h}_{I^v(n),\bar s_v,g}$.
\end{nota}

Let then denote by  $\mathfrak m^v$ the multiset of
Hecke eigenvalues given by $\mathfrak m$ but outside $v$ and 
introduce for $\Pi_h$ any representation of $GL_h(F_v)$
$$H^i(\sh_{I^v(\oo),\bar s_v,\overline{1_{h}}}^{\geq h}, V_{\xi,\overline \Zm_l})_{\mathfrak m^v} \otimes \Pi_h:=
\lim_{\atop{\longrightarrow}{n}} H^i(\sh_{I^v(n),\bar s_v,\overline{1_{h}}}^{\geq h},
V_{\xi,\overline \Zm_l})_{\mathfrak m^v} \otimes \Pi_h,$$
as a representation of $GL_h(F_v) \times GL_{d-h}(F_v)$, where $g \in GL_h(F_v)$
acts on $\Pi_h$ as well as on $H^i(\sh_{I^v(n),\bar s_v,\overline{1_{h}}}^{\geq h},
V_{\xi,\overline \Zm_l})_{\mathfrak m^v}$ through the determinant map $\det: GL_h(F_v)
\twoheadrightarrow F_v^\times$. Note moreover that the unipotent radical of
$P_{h,d}(F_v)$ acts trivially on these cohomology groups. 
We then introduce their induced version
$$H^i(\sh^{\geq h}_{I^v(\oo),\bar s_v},\Pi_h \otimes V_{\xi,\overline \Zm_l})_{\mathfrak m^v}
\simeq \ind_{P_{h,d}(F_v)}^{GL_d(F_v)} H^i(\sh_{I^v(\oo),\bar s_v,\overline{1_h}}^{\geq h}, V_{\xi,\overline \Zm_l})_{\mathfrak m^v} \otimes \Pi_h.$$

More generally, with the notations of \cite{boyer-invent2}, replace now the trivial representation by
 an irreducible cuspidal representation $\pi_v$ of $GL_g(F_v)$ for some $1 \leq g \leq d$.
 
 \begin{notas} 
Let $1 \leq t \leq s:=\lfloor d/g \rfloor$ and $\Pi_t$ any representation of 
 $GL_{d-tg}(F_v)$. We then denote by
$$\widetilde{HT}_1(\pi_v,\Pi_t):=\LC(\pi_v[t]_D)_{\overline{1_{tg}}} 
\otimes \Pi_t \otimes \Xi^{\frac{tg-d}{2}}$$ 
the Harris-Taylor local system on the Newton stratum $\sh^{=tg}_{I,\bar s_v,\overline{1_{tg}}}$ where 
\begin{itemize}
\item $\LC(\pi_v[t]_D)_{\overline{1_{tg}}}$ is defined thanks to
Igusa varieties attached to the representation $\pi_v[t]_D$ of the division algebra of dimension
$(tg)^2$ over $F_v$ associated to $\st_t(\pi_v)$ by the Jacquet-Langlands correspondence,

\item $\Xi:\frac{1}{2} \Zm \longrightarrow \overline \Zm_l^\times$ defined by $\Xi(\frac{1}{2})=q^{1/2}$.
\end{itemize}
We also introduce the induced version
$$\widetilde{HT}(\pi_v,\Pi_t):=\Bigl ( \LC(\pi_v[t]_D)_{\overline{1_{tg}}} 
\otimes \Pi_t \otimes \Xi^{\frac{tg-d}{2}} \Bigr) \times_{P_{tg,d}(F_v)} GL_d(F_v),$$
where the unipotent radical of $P_{tg,d}(F_v)$ acts trivially and the action of
$$(g^{\oo,v},\left ( \begin{array}{cc} g_v^c & * \\ 0 & g_v^{et} \end{array} \right ),\sigma_v) 
\in G(\Am^{\oo,v}) \times P_{tg,d}(F_v) \times W_v$$ 
is given
\begin{itemize}
\item by the action of $g_v^c$ on $\Pi_t$ and 
$\deg(\sigma_v) \in \Zm$ on $ \Xi^{\frac{tg-d}{2}}$, and

\item the action of $(g^{\oo,v},g_v^{et},\val(\det g_v^c)-\deg \sigma_v)
\in G(\Am^{\oo,v}) \times GL_{d-tg}(F_v) \times \Zm$ on $\LC_{\overline \Qm_l}
(\pi_v[t]_D)_{\overline{1_h}} \otimes \Xi^{\frac{tg-d}{2}}$.
\end{itemize}
We also introduce
$$HT(\pi_v,\Pi_t)_{\overline{1_{tg}}}:=\widetilde{HT}(\pi_v,\Pi_t)_{\overline{1_{tg}}}[d-tg],$$
and the perverse sheaf
$$P(t,\pi_v)_{\overline{1_{tg}}}:=j^{=tg}_{1,!*} HT(\pi_v,\st_t(\pi_v))_{\overline{1_{tg}}} 
\otimes \Lm(\pi_v),$$
and their induced version, $HT(\pi_v,\Pi_t)$ and $P(t,\pi_v)$, where 
$$j^{=h}=i^h \circ j^{\geq h}:\sh^{=h}_{I,\bar s_v} \hookrightarrow
\sh^{\geq h}_{I,\bar s_v} \hookrightarrow \sh_{I,\bar s_v}$$ 
and $\Lm^\vee$, the dual of $\Lm$, is the local Langlands correspondence.
Finally we will also use the indice $\xi$ in the notations, for example $HT_\xi(\pi_v,\Pi_t)$, 
when we twist the sheaf with $V_{\xi,\overline \Zm_l}$.
\end{notas}

With the previous notations, from (\ref{eq-red-tau}), we deduce the following equality in the
Grothendieck group of Hecke-equivariant local systems
\addtocounter{thm}{1}
\begin{equation} \label{eq-chgt-cuspi}
m(\varrho)l^{u}\Bigl [ \Fm \LC_{\xi,\overline \Zm_l}(\pi_{v,u}[t]_D) \Bigr ] =  
\Bigl [ \Fm \LC_{\xi,\overline \Zm_l}(\pi_{v,-1}[tm(\varrho)l^u]_D) \Bigr ].
\end{equation}

We want now to focus on the perverse Harris-Taylor sheaves. Note first that over 
$\overline \Zm_l$, there are two notions of intermediate extension associated to 
the two classical $t$-structures $p$ and $p+$.
So for every $\pi_v \in  \cusp_{\varrho}$ of $GL_g(F_v)$ and $1 \leq t \leq d/g$, we can define:
\addtocounter{thm}{1}
\begin{equation} \label{eq-bimorphism}
\lexp p j^{=tg}_{!*} HT(\pi_{v,},\Pi_t) \htarrow_+ \lexp {p+} j^{=tg}_{!*} HT(\pi_{v},\Pi_t),
\end{equation}
the symbol $\htarrow_+$ meaning bimorphism, i.e. both a monomorphism and epimorphism, so that
the cokernel for the $t$-structure $p$ (resp. the kernel for $p+$) has support in
$\sh^{\geq tg+1}_{I,\bar s_v}$. When $\pi_v$ is a character, i.e. when $g=1$, the associated
bimorphims are isomorphisms, as explained in the following lemma, 
but in general there are not.

\begin{lemma} \label{lem-ext0} 
With the previous notations, we have an isomorphism
$$\lexp p j^{\geq h}_{\overline{1_h},!*}  HT(\chi_v,\Pi_h) 
\simeq \lexp {p+} j^{\geq h}_{\overline{1_h},!*} HT(\chi_v,\Pi_h).$$
\end{lemma}

\begin{proof}
Recall that $\sh^{\geq h}_{I,\bar s_v,\overline{1_h}}$ is smooth over 
$\spec  \overline \Fm_p$. As, up to a modification of the action of the fondamental
group through the character $\chi_v$, we have
$$HT(\chi_v,\Pi_h)_{\overline{1_h}}[h-d]= (\overline \Zm_l)_{|\sh^{\geq h}_{I,\bar s_v,\overline{1_h}}} \otimes \Pi_h.$$
Then $HT(\chi_v,\Pi_h)_{\overline{1_h}}$ is perverse for the two $t$-structures with
$$i^{h \leq +1,*}_{\overline{1_h}} HT(\chi_v,\Pi_h)_{\overline{1_h}} \in 
\lexp p \DC^{< 0} \hbox{ and }
i^{h \leq +1,!}_{\overline{1_h}} HT(\chi_v,\Pi_h)_{\overline{1_h}} \in 
\lexp {p+} \DC^{\geq 1}.$$
\end{proof}

\rem One of the main result of \cite{boyer-duke}, is the fact that
the previous lemma holds for any $\pi_v \in \cusp_\varrho(-1)$.

As explained in the introduction, with the hypothesis on the order of $q_v$ modulo $l$
which is supposed to be strictly greater than $d$, for $\varrho$ the trivial representation, 
we do not need to bother about the representations 
$\pi_v \in \cusp_\varrho(u)$ for $u \geq 0$,
 cf. the remark after \ref{nota-mvarrho}.

\subsection{Filtrations of the nearby cycles perverse sheaf}
\label{para-fil-psi}

 Let denote by\footnote{We decide not to add $I$ in the list of indexes.} 
$$\Psi_{v}:=R\Psi_{\eta_v}(\overline \Zm_l[d-1])(\frac{d-1}{2})$$
the nearby cycles autodual free perverse sheaf on the geometric special fiber $\sh_{I,\bar s_v}$
of $\sh_I$. We also denote by $\Psi_{\xi,v}:=\Psi_{v} \otimes V_{\xi,\overline \Zm_l}$.

Using the Newton stratification and following the constructions of \cite{boyer-FT},
we can define a $\overline \Zm_l$-filtration $\Fill^\bullet(\Psi_v)$ whose graded parts are free, 
isomorphic to some free perverse Harris-Taylor sheaf. Moreover,
by denoting $\scusp_{\overline \Fm_l}(g)$ for the set of inertial equivalence 
classes of irreducible $\overline \Fm_l$-supercuspidal representations of 
$GL_g(F_v)$,  in \cite{boyer-duke} 
proposition 3.1.3, we proved the following splitting
\addtocounter{thm}{1}
\begin{equation} \label{eq-psi-dec}
\Psi_{v} \simeq \bigoplus_{g=1}^d 
\bigoplus_{\varrho \in \scusp_{\overline \Fm_l}(g)} \Psi_{\varrho},
\end{equation}
with the property that
the irreducible sub-quotients of 
$$\Psi_{\varrho} \otimes_{\overline \Zm_l} \overline \Qm_l \simeq
\bigoplus_{\pi_v \in \cusp_\varrho} \Psi_{\pi_v}$$
are exactly the perverse Harris-Taylor sheaves, of level $I$, 
associated to an irreducible cuspidal $\overline \Qm_l$-representation
of some $GL_g(F_v)$ such that the supercuspidal support of the modulo $l$ reduction of
$\pi_v$ is a segment associated to the inertial class $\varrho$.

\rem In \cite{boyer-duke}, we proved that if you always use the adjunction maps $j^{=h}_! j^{=h,*}
\rightarrow \Id$ then all the previous graded parts of $\Psi_{\varrho}$ 
are isomorphic to $p$-intermediate extensions. In the following we will only consider the 
case where $\varrho$ is a character in which case, cf. lemma \ref{lem-ext0}, the $p$ and $p+$ intermediate extensions 
associated to character $\chi_{v,-1} \in \cusp_\varrho(-1)$, coincide. 
Note that in the following
we will not use the results of \cite{boyer-duke}.

\medskip

Denoting by $\grr^k(\Psi_{\xi,v}):=\Fill^k(\Psi_{\xi,v})/\Fill^{k-1}(\Psi_{\xi,v})$,
we then have a spectral sequence
\addtocounter{thm}{1}
\begin{equation} \label{eq-ss1}
E_1^{p,q}=H^{p+q}(\sh_{I,\bar s_v},\grr^{-p}(\Psi_{\xi,v})) \Rightarrow 
H^{p+q}(\sh_{I,\bar \eta_v},V_{\xi,\overline \Zm_l}),
\end{equation}
where we recall that
\addtocounter{thm}{1}
\begin{multline} \label{eq-grppsi}
\lexp p j^{=tg}_{!*} HT_\xi(\pi_v,\st_t(\pi_v))(\frac{1-t+2i}{2}) \htarrow_+  
\grr^k(\Psi_{\xi,v}) \htarrow_+ \\
\lexp {p+} j^{=tg}_{!*} HT_\xi(\pi_v,\st_t(\pi_v))(\frac{1-t+2i}{2}),
\end{multline}
for some irreducible cuspidal representation $\pi_v$ of $GL_g(F_v)$ with $1 \leq t \leq d/g$ and
$0 \leq i \leq \lfloor d/g \rfloor -1$.
%

Let consider now the filtration of stratification of $\Psi_{\xi,\varrho}$ constructed using the
adjunction morphisms $j^{=t}_! j^{=t,*}$ as in \cite{boyer-torsion}
$$\Fil^0_!(\Psi_{\xi,\varrho}) \harrow \Fil^1_!(\Psi_{\xi,\varrho}) \harrow 
\Fil^{2}_!(\Psi_{\xi,\varrho})
\harrow \cdots \harrow \Fil^{d}_!(\Psi_{\xi,\varrho})$$
where the symbol $\harrow$ means a monomorphism such that the cokernel
is torsion free which here means that
$\Fil^{t}_!(\Psi_{\xi,\varrho})$ is the saturated image of 
$j^{=t}_!j^{=t,*} \Psi_{\xi,\varrho} \longrightarrow \Psi_{\xi,\varrho}$.  
We then denote by $\gr^k_!(\Psi_{\xi,\varrho})$ the graded parts and
\addtocounter{thm}{1}
\begin{equation} \label{eq-ss2}
E_{!,\varrho,1}^{p,q}=H^{p+q}(\sh_{I,\bar s_v},\gr^{-p}_!(\Psi_{\xi,\varrho})) \Rightarrow 
H^{p+q}(\sh_{I,\bar \eta_v},\Psi_{\xi,\varrho}).
\end{equation}

\rem Over $\overline \Qm_l$, in \cite{boyer-torsion} we prove that
$\Fil^k_!(\Psi_{\xi,\varrho}) \otimes_{\overline \Zm_l} \overline \Qm_l$
is $\ker N^k$ where $N$ is the monodromy operator at $v$. Moreover
there is only one filtration of stratification of $\gr^k_!(\Psi_{\xi,\varrho})$
which we denote simply by $\Fil^\bullet (\gr^k_!(\Psi_{\xi,\varrho}))$ in
the following. Over $\overline \Qm_l$, 
$\Fil^\bullet (\gr^k_!(\Psi_{\xi,\pi_v}))$
coincides with the filtration through the iterated image of $N$, i.e.
$\gr^r(\gr^k_!(\Psi_{\xi,\pi_v}))=\im N^r \cap \ker N^k$ so that
we recover the usual bi-filtration of monodromy of \cite{boyer-invent2}.

\section{Irreducibility implies freeness}
\label{para-proof}

Recall that we argue by absurdity assuming there exists non trivial 
cohomology classes in some of the 
$H^i(\sh_{I,\bar \eta},V_{\xi,\overline \Zm_l})_{\mathfrak m}$.
The strategy is then to choose a place $v \in \spl_I$ such that the order of $q_v$ modulo
$l$ is strictly greater than $d$, and to compute the middle cohomology group
 $H^{d-1}(\sh_{I,\bar \eta},V_{\xi,\overline \Zm_l})_{\mathfrak m}$ through the spectral
 sequence of vanishing cycles. This spectral sequence gives us in particular a filtration of the
 free quotient of $H^{d-1}(\sh_{I,\bar \eta},V_{\xi,\overline \Zm_l})_{\mathfrak m}$ 
We will in fact consider different sorts of level $I$ relatively to the place $v$ and another
one denoted $w$ verifying the same hypothesis as $v$:
\begin{itemize}
\item either with infinite level at $v$ or with $I_v$ some particular Iwahori subgroup;

\item either maximal or infinite at $w$.
\end{itemize}

\rem From now on, the localization at $\mathfrak m$, means that we prescribe
the modulo $l$ Satake's parameters as usual, but outside $\{ v,w \}$.

By hypothesis, there exists non trivial torsion cohomology classes in level $I$
with $I_v$ and $I_w$ maximal. As moreover we supposed the order of both $q_v$ and
$q_w$ is strictly greater than $d$, then the functors of invariants by any open compact
subgroups either at $v$ or $w$, is exact. Thus this allows us to argue similarly with
all the mentioned level, cf. lemma \ref{lem-iIv}.
Let now explain the main steps of the following sections.

(a) Following the arguments of the previous section, 
we first analyse the torsion cohomology
classes of Harris-Taylor perverse sheaves with infinite level at $v$, and we deduce, 
cf. lemma \ref{lem-rem-after1}, that, as $\overline \Fm_l$-representations
of $GL_d(F_v)$, 
irreducible sub-quotients of the $l$-torsion
of their cohomology in infinite level at $v$, with highest non degeneracy level, appears in degree $0,1$.

(b) In \S \ref{para-generic}, considering always infinite level at $v$, 
we analyse the torsion
cohomology classes of the graded parts $\gr^t_!(\Psi_\varrho)$ of the filtration of
stratification constructed using the adjunction property $j^{=t}_! j^{=t,*} \rightarrow \Id$.
We then deduce, cf. lemma \ref{lem-generic-sq}, that the $l$-torsion of 
$H^i(\sh_{I^v(\oo),\bar s_v},V_{\xi,\overline \Zm_l})_{\mathfrak m}$ 
does not have, as a $\overline \Fm_l$-representation of $GL_d(F_v)$,
any irreducible generic subquotient whose supercuspidal support is made of
characters.

(c) In section \ref{para-modification}, we obtain two fondamental results.
\begin{itemize}
\item First, cf. lemma \ref{lem-important},  
under the hypothesis that there exists non trivial torsion cohomology classes, we show that
the graded pieces $\Gamma_k$ of the filtration of the free quotient of
$H^0(\sh_{I^v(\oo),\bar s_v},V_{\xi,\overline \Zm_l})_{\mathfrak m}$
are not always given by the lattice of
$H^0(\sh_{I,\bar s_v},P_\xi(t,\chi_v)(\frac{1-t+2\delta}{2}))_{\mathfrak m} 
\otimes_{\overline \Zm_l} \overline \Qm_l$ given by the integral cohomology of
$P_\xi(t,\chi_v)$. Roughly there exists some $k$ and a short exact
sequence $\Gamma_0 \hookrightarrow \Gamma_k \twoheadrightarrow T$ where $\Gamma_0$
is the lattice given by the integral cohomology of the associated Harris-Taylor perverse sheaf, 
and $T$ is non trivial and torsion.

\item We then play with the action of $GL_d(F_w)$ by allowing infinite level at $w$. The main
observation at the end of the section, cf. proposition \ref{prop-important},
is that as a $\overline \Fm_l$-representation
of $GL_d(F_w)$, all the irreducible sub-quotients of the $l$-torsion of $T$, up to multiplicities,
are also sub-quotients of the $l$-torsion of the global cohomology. In particular,
as $v$ and $w$ are playing symmetric roles, these sub-quotients are not generic,
cf. corollary \ref{coro-absurd}.
\end{itemize}
 
(d) In \S \ref{para-final}, the last step is to prove, under the absurd hypothesis that
there exists non trivial torsion cohomology classes while
$\overline \rho_{\mathfrak m}$ being irreducible, then $S_{\mathfrak m}(v)$
contains a full set $\{ \lambda q_v^n: n \in \Zm \}$ which is of order the order
of $q_v$ modulo $l$. As this order is supposed to be strictly greater than $d$,
this is absurd. For more insight on the strategy to prove this fact using the previous
properties about lattices, we refer to
the introduction of \S \ref{para-final}.

\subsection{Torsion classes for Harris-Taylor perverse sheaves.}

We focus on the torsion in the cohomology groups of the Harris-Taylor perverse
sheaves $P_\xi(\chi_v,t)$ when the level at $v$  
is infinite. 

\begin{nota} \label{nota-mvw}
We will denote by $I^{v}(\oo) \in \IC$ a finite\footnote{and morally infinite at $v$} 
level outside $v$, and we also denote by
$\mathfrak m$ the maximal ideal of $\Tm^{S \cup \{ v \} }_{\xi}$ 
associated to $\mathfrak m$, i.e. we do not prescribe the modulo $l$ Satake's at $v$.
Let also denote
$$H^i(\sh_{I^v(\oo),\bar s_v},\overline \Zm_l)_{\mathfrak m}:= 
\lim_{\atop{\rightarrow}{I_v}} 
H^i(\sh_{I,\bar s_v},\overline \Zm_l)_{\mathfrak m},$$
which can be viewed as a $\overline \Zm_l[GL_d(F_v)]$-module.
\end{nota}

\begin{prop} We have
the following resolution of $\lexp p j_{!*}^{=t} HT(\chi_{v},\Pi_t)$
\addtocounter{thm}{1}
\begin{multline} \label{eq-resolution0}
0 \rightarrow j_!^{=d} HT(\chi_{v},\Pi_t \{ \frac{t-s}{2} \} ) \times 
\speh_{d-t}(\chi_{v}\{ t/2 \} ))
 \otimes \Xi^{\frac{s-t}{2}} \longrightarrow \cdots  \\
\longrightarrow j_!^{=t+1} HT(\chi_{v},\Pi_t \{ -1/2 \}  \times \chi_{v} \{ t/2 \} ) 
\otimes \Xi^{\frac{1}{2}} \longrightarrow \\ j_!^{=t} HT(\chi_{v},\Pi_t) 
\longrightarrow  \lexp p j_{!*}^{=t} HT(\chi_{v},\Pi_t) \rightarrow 0.
\end{multline}
\end{prop}

Note that
\begin{itemize}
\item as this resolution is equivalent to the computation of the sheaves 
cohomology groups
of $\lexp p j_{!*}^{=h} HT(\chi_v,\st_h(\chi_v)) $ as explained for example in
\cite{boyer-duke} proposition B.1.5 of appendice B, then, 
over $\overline \Qm_l$, it follows from the main results of \cite{boyer-invent2}.

\item Over $\overline \Zm_l$, as every terms are free perverse sheaves, then
all the maps are necessary strict. 

\item This resolution, for a a general supercuspidal representation
with supercuspidal modulo $l$ reduction, is one of the main result 
of \cite{boyer-duke} \S 2.3.
\end{itemize}

\begin{proof}
For the case of a character $\chi_v$ as above, the argument is almost 
obvious. Indeed as the strata
$\sh^{\geq h}_{I^v,\bar s_v,1}$ are smooth, then, cf. the proof of the lemma
\ref{lem-ext0}, the constant sheaf, up to shift, is perverse and
so equals to the intermediate extension of the constant sheaf, shifted by $d-h$, 
on $\sh^{=h}_{I^v,\bar s_v,1}$. In particular its sheaves cohomology groups
are well known so that the resolution is completely obvious for 
$\lexp p j_{\overline{1_h},!*}^{=h} HT_{\overline{1_h}}(\chi_v,\st_h(\chi_v))$
if one remember that $\speh_i(\chi_v)$ is just the character $\chi_v \circ \det$
of $GL_i(F_v)$.

The stated resolution is then simply the induced version of the resolution
of $\lexp p j^{=h}_{\overline{1_h},!*} HT_{\overline{1_h}}(\chi_v,\st_h(\chi_v))$:
recall that a direct sum of intermediate extensions is still an intermediate 
extension.
%
\end{proof}

By the adjunction property, the map
\addtocounter{thm}{1}
\begin{multline} \label{eq-map1}
j_!^{=t+\delta} HT(\chi_{v},\Pi_t \{ \frac{-\delta}{2} \} ) \times \speh_{\delta}
(\chi_{v}\{ t/2 \} )) \otimes \Xi^{\delta/2} \\
\longrightarrow j_!^{=t+\delta-1} HT(\chi_{v},\Pi_t \{ \frac{1-\delta}{2} \} ) \times 
\speh_{\delta-1}(\chi_{v}\{ t/2 \} )) \otimes \Xi^{\frac{\delta-1}{2}}
\end{multline}
is given by 
\addtocounter{thm}{1}
\begin{multline}\label{eq-map2}
HT(\chi_{v},\Pi_t \{ \frac{-\delta}{2} \}  \times \speh_{\delta}(\chi_{v}\{ t/2 \} )) 
\otimes \Xi^{\delta/2} \longrightarrow \\
\lexp p i^{t+\delta,!}  j_!^{=t+\delta-1} HT(\chi_{v},\Pi_t \{ \frac{1-\delta}{2} \} ) \times 
\speh_{\delta-1}(\chi_{v}\{ t/2 \} )) \otimes \Xi^{\frac{\delta-1}{2}}
\end{multline}
 We have then
\addtocounter{thm}{1}
\begin{multline} \label{eq-map3}
\lexp p i^{t+\delta,!}  j_!^{=t+\delta-1} HT(\chi_{v},\Pi_t \{ \frac{1-\delta}{2} \} ) \times 
\speh_{\delta-1}(\chi_{v}\{ t/2 \} )) \otimes \Xi^{\frac{\delta-1}{2}} \\ \simeq 
HT \Bigl ( \chi_{v}, \Pi_t \{ \frac{1-\delta}{2} \} )  \times \bigl ( \speh_{\delta-1}(\chi_{v} \{ -1/2 \})
 \times \chi_{v} \{\frac{\delta-1}{2} \}  \bigr ) \{ t/2 \}  \Bigr ) \otimes \Xi^{\delta/2}.
\end{multline}
Indeed one can compute $\lexp p i^{h+1,!} j_{!}^{=h} HT(\chi_v,\Pi_h)$ through
the spectral sequence associated to the exhaustive filtration of stratification of
$j_{!}^{=h} HT(\chi_v,\Pi_h)$ 
\addtocounter{thm}{1}
\begin{multline} \label{eq-filj}
(0)=\Fil^0(\chi_{v} ,h) \hookrightarrow \Fil^{-d}(\chi_{v} ,h) \hookrightarrow \cdots
\\
\hookrightarrow \Fil^{-h}(\chi_{v} ,h)=j^{=h}_{!} HT(\chi_{v} ,\Pi_h)
\end{multline}
with graded parts, using \ref{lem-ext0} and \cite{boyer-torsion},
$$\gr^{-k}(\chi_{v},h) \simeq \lexp p j^{=k}_{!*} 
HT(\chi_{v},\Pi_h \{\frac{h-k}{2} \} \otimes \st_{k-h}(\chi_{v}\{h/2 \} ))(\frac{h-k}{2}).$$ 
As remarked before the sheaf cohomology groups of 
$$i^{h+1,*} \lexp p j_{!*}^{h+k}  HT(\chi_v,\Pi_h \{ -k/2 \} \times \st_k(\chi_v)(h/2))$$
are torsion free, so, by the Grothendieck-Verdier duality, the same is true for
$$i^{h+1,!} \lexp p j_{!*}^{h+k}  HT(\chi_v,\Pi_h \{ -k/2 \}\times \st_k(\chi_v)(h/2)).$$
The statement follows then from the fact that, over
$\overline \Qm_l$, the previous spectral sequence degenerates at $E_1$.

\rem This property is also true when we replace the character $\chi_v$ by
any irreducible cuspidal representation $\pi_v$, cf. \cite{boyer-duke}.

\noindent \emph{Fact}.
In particular, up to homothety, the map (\ref{eq-map3}), and so those of (\ref{eq-map2}), is 
unique. Finally as the map of (\ref{eq-resolution0}) are strict, the given maps (\ref{eq-map1}) 
are uniquely determined, that is if we forget the infinitesimal parts, these maps are independent 
of the chosen $t$ in (\ref{eq-resolution0}).

We want now to copy the arguments of \S \ref{para-max}.

\begin{nota} \label{nota-ih}
For every $1 \leq h \leq d$, let denote by $i_{I^v}(h)$ the smaller index $i$ such that 
$H^i(\sh_{I^{v}(\oo),\bar s_v},\lexp p j^{=h}_{!*} HT_\xi(\chi_{v} ,\Pi_h ))_{\mathfrak m}$
has non trivial torsion: if it doesn't exists then set $i_{I^v}(h)=+\oo$. 
\end{nota}

\rem By duality, as 
$\lexp p j_{!*}^{=h}=\lexp {p+} j^{=h}_{!*}$ for Harris-Taylor local systems associated to a
character, note that when $i_I(h)$ is finite then $i_{I^v}(h) \leq 0$. 
Note also that, using the classical determinant map, it is in fact independent of the character
$\chi_v$. However we do not need this fact in the following.

\begin{nota} \label{nota-h0}
Suppose there exists $I \in \IC$ such that
there exists $1 \leq h \leq d$ with $i_{I^v}(h)$ finite and denote by $h_0(I^v)$ 
the bigger such $h$.
\end{nota}

\begin{lemma} \label{lem-torsion-min}
For $1 \leq h \leq h_0(I^v)$ then $i_{I^v}(h)=h-h_0(I^v)$. Moreover
$\frob_v$ acts by $\chi_v(\frob_v) q_v^{\frac{h_0(I^v)+1-h}{2}}$.
\end{lemma}

\begin{proof}
Note first that for every $h_0(I^v) \leq h \leq s$, then the cohomology groups of 
$j^{=h}_! HT_\xi(\chi_{v},\Pi_h)$ are torsion free. 
The $\xi$-associated spectral sequence associated to the filtration
(\ref{eq-filj}), localized at $\mathfrak m$, is then concentrated in 
middle degree and torsion free.

Consider then the spectral sequence associated to the resolution (\ref{eq-resolution0}): its 
$E_1$ terms are torsion free and it degenerates at $E_2$. 
As by hypothesis the aims of this spectral sequence is free
and equals to only one $E_2$ terms, we deduce that all the maps
\addtocounter{thm}{1}
\begin{multline} \label{eq-map1-coho0}
H^0 \bigl (\sh_{I^{v}(\oo),\bar s_v},j_!^{=h+\delta} HT_\xi(\chi_{v},\Pi_h \{ \frac{-\delta}{2} \} ) \times 
\speh_{\delta} (\chi_{v}\{ t/2 \} )) \otimes \Xi^{\delta/2} \bigr )_{\mathfrak m} 
\longrightarrow \\ H^0 \bigl (\sh_{I^{v}(\oo),\bar s_v},
j_!^{=h+\delta-1} HT_\xi(\chi_{v},\Pi_h \{ \frac{1-\delta}{2} \} ) \\ \times 
\speh_{\delta-1}(\chi_{v}\{ t/2 \} )) \otimes \Xi^{\frac{\delta-1}{2}} \bigr )_{\mathfrak m}
\end{multline}
are strict. Then from the previous fact stressed after (\ref{eq-map3}), this property
remains true when we consider the associated spectral sequence for $1 \leq h' \leq h_0$.

Consider now  $h=h_0(I^v)$ 
where we know the torsion to be non trivial. From what was observed above
we then deduce that the map
\addtocounter{thm}{1}
\begin{multline} \label{eq-map1-coho}
H^0 \bigl (\sh_{I^{v}(\oo),\bar s_v},j_!^{=h_0(I^v)+1} HT_\xi(\chi_{v},\Pi_{h_0(I^v)} \{ \frac{-1}{2} \} ) 
\times \chi_{v}\{ h_0(I^v)/2 \} )) \otimes \Xi^{1/2} \bigr )_{\mathfrak m} \\
\longrightarrow  H^0 \bigl (\sh_{I^{v}(\oo),\bar s_v},
j_!^{=h_0(I^v)} HT_\xi(\chi_{v},\Pi_{h_0(I^v)}) \bigr )_{\mathfrak m}
\end{multline}
has a non trivial torsion cokernel so that $i_I(h_0(I^v))=0$.

Finally for any $1 \leq h \leq h_0$, the map like (\ref{eq-map1-coho}) for $h+\delta-1 < h_0$
are strict so that the $H^i(\sh_{I^{v}(\oo),\bar s_v},\lexp p j^{=h}_{!*} 
HT_\xi(\chi_{v},\Pi_h))_{\mathfrak m}$ are zero for $i < h-h_0$ while when 
$h+\delta-1=h_0$ its cokernel has non trivial torsion which gives
then a non trivial torsion class in $H^{h-h_0}(\sh_{I^{v}(\oo),\bar s_v},\lexp p j^{=h}_{!*} 
HT_\xi(\chi_{v},\Pi_h))_{\mathfrak m}$.
\end{proof}

\begin{lemma} \label{lem-iIv}
With the notation of \ref{prop-h0I}, we have $h_0(I^v) \geq h_0(I)$.
\end{lemma}

\begin{proof}
Consider the previous map (\ref{eq-map1-coho}) by replacing $h_0(I^v)$ by $h_0(I)$.
As by hypothesis the order of $q_v$ modulo $l$ is strictly greater than $d$, 
then the 
pro-order of the local component $I_v$ of $I$ at $v$,
is invertible modulo $l$, so that the functor of invariants under $I_v$ is exact.
Note then that, as the $I_v$-invariants of the map (\ref{eq-map1-coho}) when
replacing $h_0(I^v)$ by $h_0(I)$, has a cokernel which is not free, then the cokernel
of (\ref{eq-map1-coho}), for $h_0(I)$, is also not free.
%
%
\end{proof}

\rem The main reason to go to infinite level at $v$ is to be able to use
the notion of level of non degeneracy which will be convenient but we will
have to deal with the possibility that $h_0(I^v)$ might be strictly greater
than $h_0(I)$ so that appear extra torsion classes with no interest for us.

From the previous proof, we also deduce that all cohomology classes of any of 
the $H^i(\sh_{I^{v}(\oo),\bar s_v},P_\xi(t,\chi_v))_{\mathfrak m}$ comes 
from the non strictness of some of
the map (\ref{eq-map1-coho}) where $\Pi_v:=\st_t(\chi_v)$. In the following we will focus
on $H^i(\sh_{I^{v}(\oo),\bar s_v},P_\xi(t,\chi_v))_{\mathfrak m}[l]$ as a 
$\overline \Fm_l$-representation of $GL_d(F_v)$. 
More precisely we are interested in irreducible such sub-quotients which have
maximal non-degeneracy level at $v$.

\begin{nota}
Let first fix such non degeneracy level $\underline{\lambda}$ for $GL_d(F_v)$
in the sense of notation \ref{nota-nondegeneracy}, which is maximal for torsion classes
in $H^0(\sh_{I^{v}(\oo),\bar s_v},P_\xi(t,\chi_v))_{\mathfrak m}[l]$ for various $1 \leq t \leq d$ and $\chi_v \in \cusp_\varrho(-1)$.
\end{nota}

\rem As mentioned after \ref{nota-ih}, for the definition of $\lambda$, you could also
consider any fixed $\chi_v \in \cusp_\varrho(-1)$. 

\begin{lemma} \label{lem-rem-after1}
Let $\varrho$ be a $\overline \Fm_l$-character of $F_v^\times$ and $\chi_v \in \cusp_{-1}(\varrho$.  
Then all $\overline \Fm_l[GL_d(F_v)]$-irreducible sub-quotients of
$H^i(\sh_{I^{v}(\oo),\bar s_v},P_\xi(t,\pi_v))_{\mathfrak m}[l]$, for $i\neq 0,1$, have a level of 
non degeneracy strictly less than $\underline{\lambda}$.
\end{lemma}

\rem Recall that we only know, a priori, that the $P(t,\pi_v)$ only verify
$$\lexp p j^{=tg}_{!*} HT(\pi_v,\st_t(\pi_v)) \htarrow_+ P(t,\pi_v) \htarrow_+
\lexp {p+} j^{=tg}_{!*} HT(\pi_v,\st_t(\pi_v)).$$

\begin{proof}
It easily follows from the observation that the level of non degeneracy of 
the modulo $l$ reduction of $\speh_h(\chi_v) \simeq \chi_v$ is strictly less than
those of the modulo $l$ reduction of $\st_h(\chi_v)$ which is irreducible as the order
of $q_v$ modulo $l$ is strictly greater than $d \geq h$.
\end{proof}

\rem The general case where $\pi_v \in \cusp_u(\varrho)$ for $u \geq 0$
is also true. One first have to deal with
$\lexp p j^{=t}_{!*} HT_{\xi,\overline \Fm_l}(\varrho,\Pi_t)[d-t]$, i.e.
if $H^i(\sh_{I^{v}(\oo),\bar s_v},\lexp p j^{=t}_{!*} HT_{\xi,\overline \Fm_l}(\varrho,\Pi_t)
[d-t])_{\mathfrak m}$
and then use the equality of proposition 2.4.2 of \cite{boyer-duke}.

%

%
%
%
%
%
%
%
%
%


\subsection{Global torsion and genericity}
\label{para-generic}

Recall that $v \in \spl$ is such that the order of $q_v$ modulo $l$ is strictly
greater than $d$. Let denote by $I^{v}$ the component of $I$ outside $v$. 
We then simply denote by $\Psi_v$ and $\Psi_{v,\xi}$, 
the inductive system of perverse sheaves indexed by the finite level $I^{v} I_v  \in \IC$
for varying $I_v$.

%
%
%

For $\pi_{v} \in \cusp_\varrho$, let denote by
$$\Fil^1_!(\Psi_\varrho) \twoheadrightarrow \Fil^1_{!,\pi_{v}}(\Psi_\varrho)$$ 
such that
$\Fil^1_{!,\pi_{v}}(\Psi_v) \otimes_{\overline \Zm_l} \overline \Qm_l \simeq \Fil^1_!(\Psi_{\pi_{v}})$
where $\Psi_{\pi_{v}}$ is the direct factor of $\Psi_v \otimes_{\overline \Zm_l} \overline \Qm_l$
associated to $\pi_{v}$, cf. \cite{boyer-torsion}. 

\rem In the following, we will mainly be concerned with the case where $\pi_v$
is a character $\chi_v$. We will then write the main statement in this case.

Recall the following resolution of $\Fil^1_{!,\chi_{v}}(\Psi_\varrho)$
\addtocounter{thm}{1}
\begin{multline} \label{eq-resolution-psi0}
0 \rightarrow j^{=d}_! HT(\chi_{v},\speh_d(\chi_{v})) \otimes \Lm (\chi_{v}(\frac{d-1}{2}))
\longrightarrow \\
 j^{=d-1}_!HT(\chi_{v},\speh_{d-1}(\chi_{v})) \otimes \Lm (\chi_{v}(\frac{d-2}{2})) 
\longrightarrow \\ \cdots \longrightarrow j^{=1}_! HT(\chi_{v},\chi_{v}) \otimes \Lm(\chi_{v})
\longrightarrow \Fil^1_{!,\chi_{v}}(\Psi_\varrho) \rightarrow 0
\end{multline}
which is proved in \cite{boyer-torsion} over $\overline \Qm_l$. I claim it is
also true over $\overline \Zm_l$. Indeed, using \ref{lem-ext0},
it is equivalent to the fact the sheaf cohomology of $\Fil^1_{!,\chi_{v}}(\Psi_\varrho)$
are torsion free which follows then from \cite{s-s}, the comparison 
theorem of Faltings-Fargues cf. \cite{fargues-faltings} and the main theorem
of \cite{fargues-monodromie}.
 
\rem In \cite{boyer-duke}, we prove the same resolution for any irreducible
cuspidal representation $\pi_v$ in place of $\chi_v$.

As in notation \ref{nota-mvw}, let
$\mathfrak m$ be the maximal ideal of $\Tm^{S \cup \{ v \} }_{\xi}$ 
associated to $\mathfrak m$.
We can then apply the arguments of the previous section so that 
$H^i(\sh_{I^{v}(\oo),\bar s_v},\Fil^1_{!,\chi_{v}}(\Psi_{\varrho,\xi}))_{\mathfrak m}$ 
has non trivial torsion for $i=1-t_0$
and with free quotient zero for $i \neq 0$. Clearly we can also repeat the same 
arguments for the other $\gr^t_!(\Psi_\varrho) \twoheadrightarrow 
\gr^t_{!,\chi_v}(\Psi_\varrho)$ with
\addtocounter{thm}{1}
\begin{multline} \label{eq-resolution-psi}
0 \rightarrow j^{=d}_! HT(\chi_{v},LT_{\chi_v}(t-1,d-t)) \otimes \Lm(\chi_{v}(\frac{d-2t+1}{2}))
\longrightarrow \\
 j^{=d-1}_!HT(\chi_{v},LT_{\chi_{v}}(t-1,d-t-1)) \otimes \Lm(\chi_{v}(\frac{d-2t}{2})) 
\longrightarrow \\ \cdots \longrightarrow j^{=t}_! HT(\chi_{v},\st_t(\chi_{v})) \otimes \Lm(\chi_{v})
\longrightarrow \gr^t_{!,\chi_{v}}(\Psi_\varrho) \rightarrow 0.
\end{multline}
Finally all the torsion cohomology classes of the 
$H^i(\sh_{I^v(\oo),\bar s_v},\gr^t_{!,\chi_v}(\Psi_\varrho))_{\mathfrak m}$ 
come from the non strictness of the maps
\addtocounter{thm}{1}
\begin{equation} \label{eq-map-hmax}
H^{0}(\sh_{I^v(\oo),\bar s_v},j^{=h+1}_! HT(\chi_v,\Pi_{h+1}))_{\mathfrak m} \longrightarrow
H^0(\sh_{I^v(\oo),\bar s_v},j^{=h}_! HT(\chi_v,\Pi_h))_{\mathfrak m}
\end{equation}
where $(\Pi_h,\Pi_{h+1})$ is of the shape $\Bigl ( LT_{\chi_v}(t-1,h-t),LT_{\chi_v}(t-1,h+1-t) \Bigr )$.

We can then copy the proof of lemma \ref{lem-torsion-min} which gives us the following
statement.

\begin{lemma}
For every $1 \leq h \leq h_0$, the number 
$i_I(h)=h-h_0$ of notation \ref{nota-ih}, is also the lowest integer $i$ so that the torsion
of $H^i(\sh_{I^{v}(\oo),\bar s_v},\gr^h_{!,\chi_{v}}(\Psi_{\varrho,\xi}) )_{\mathfrak m}$ is non zero.
\end{lemma}

\begin{lemma} \label{lem-generic-sq}
As a $\overline \Fm_l[GL_d(F_v)]$-module, for every $i$, the $l$-torsion of
$H^i(\sh_{I^v(\oo),\bar s_v},V_{\xi,\overline \Zm_l})_{\mathfrak m}$,
does not have an irreducible generic sub-quotient
whose cuspidal support is made of characters.
\end{lemma}

\rem Note that when the order of $q_v$ modulo $l$ is strictly greater than $d$, then
there is no difference between cuspidal or supercuspidal support made of characters.

\begin{proof}
Recall first that, as by hypothesis $\overline{\rho_{\mathfrak m}}$ is irreducible, the
$\overline \Qm_l$-version of the spectral sequence (\ref{eq-ss2}) degenerates at $E_1$ so
that in particular all the torsion cohomology classes appear in the $E_1$ terms. 
As we are only interested in representations with cuspidal support made of characters,
we only have to deal with the perverse sheaves $P(t,\chi_v)$ so that the result
follows from the previous maps (\ref{eq-map-hmax}) and the fact that for
any $r>0$, the modulo $l$ reduction of $LT_{\chi_v}(t-1,r)$ does not admit any irreducible
generic sub-quotient.
\end{proof}

\rem We could also prove the same result without restriction on the cuspidal support but
then we would have to deal with the problem mentioned in the remark after lemma
\ref{lem-rem-after1} which is the main subject of \cite{boyer-duke}.

\subsection{Torsion and modified lattices}
\label{para-modification}

Recall that we argue by absurdity, assuming there exists $I_0 \in \IC$
unramified at the place $v$, such that the torsion of some of
the $H^i(\sh_{I_0,\bar s_v},V_{\xi,\overline \Zm_l})_{\mathfrak m}$ is non
zero. We then denote $h_0:=h_0(I_0)$, cf. proposition \ref{prop-h0I}.

We want to study the free quotient of 
$H^0(\sh_{I,\bar s_v},V_{\xi,\overline \Zm_l})_{\mathfrak m}$, for some particular
level $I$, through the spectral sequence of vanishing cycles. We first focus 
on the cohomology of $\gr^{h_0}_{!,\chi_v}(\Psi_{\varrho,\xi})$ for a character
$\chi_v \in \cusp_{-1}(\varrho)$ which is by definition a quotient of
$\gr^{h_0}_{!}(\Psi_{\varrho,\xi})$.
To do so, consider first the filtration constructed in \cite{boyer-torsion}
$$\Fil^{d-h_0} ( \gr^{h_0}_{!,\chi_v}(\Psi_{\varrho,\xi})) \subset
\cdots \subset \Fil^0(\gr^{h_0}_{!,\chi_v}(\Psi_{\varrho,\xi})),$$
with successive free graded parts
$\gr^i(\gr^{h_0}_{!,\chi_v}(\Psi_{\varrho}))$ which, for $i \geq 0$, 
is a $\overline \Zm_l$-structure of
the $\overline \Qm_l$-perverse sheaf $P(h_0+i,\chi_v)  (\frac{1-h_0+i}{2})$.

For any finite level $I \in \IC$,
we then now introduce two $\overline \Zm_l$-lattices of
\addtocounter{thm}{1}
\begin{equation} \label{eq-def-sq}
H^0(\sh_{I,\bar s_v},P_\xi(h_0+i,\chi_v))_{\mathfrak m} 
\otimes_{\overline \Zm_l} \overline \Qm_l.
\end{equation}
\begin{itemize}
\item The first one denoted by $\Gamma_{\xi,\chi_v,\mathfrak m}(I,h_0+i)$ is given by the free integral cohomology: recall that as the order of $q_v$ modulo $l$
is supposed to be strictly greater than $d$ then the modulo $l$ reduction of
$\st_{h_0+i}(\chi_v)$ and that of $\chi_v[t]_D$, 
remains irreducible so that, up to homothety, there is an unique
stable lattice of $P_\xi(h_0+i,\chi_v)$.

\item The spectral sequence associated to the previous filtration of 
$\gr^{h_0}_{!,\chi_v}(\Psi_{\xi,\varrho})$, provides a filtration of $H^0_{free}
(\sh_{I,\bar s_v},\gr^{h_0}_{!,\chi_v}(\Psi_{\xi,\varrho}))_{\mathfrak m}$
and $\Gamma_{\xi,\chi_v,!,\mathfrak m}(I,h_0+i,h_0)$ is then the lattice of the
sub-quotient in this filtration corresponding to (\ref{eq-def-sq}). 
\end{itemize}

By construction we have $\Gamma_{\xi,\chi_v,\mathfrak m}(I,h_0+1) 
\hookrightarrow \Gamma_{\xi,\chi_v,!,\mathfrak m}(I,h_0+1,h_0)$ 
but the cokernel of torsion might be non trivial due to torsion
in the remaining of the $E_\oo$ terms.

\begin{nota} \label{nota-iwh0}
Let denote by
\begin{multline*} 
\Iw_v(h_0):= 
\bigl \{ g \in GL_d(\OC_v) \hbox{ such that } \\ (g \mod \varpi_v) \in
P_{1,2,\cdots,h_0,d}(\kappa(v)); \bigr \}.
\end{multline*}
\end{nota}

\rem For $h \geq h_0+1$, then $LT_{\chi_v}(h,d-h-1)$ does not have non trivial
vector invariant by $\Iw_v(h_0)$. Moreover for $\pi_v$ an irreducible
representation of $GL_{d-h}(F_v)$ with $h \geq h_0+1$, 
then $LT_{\chi_v}(h_0-1,h-h_0) \times \pi_v$
admits non trivial invariants vectors by $\Iw_v(h_0)$ if and only if
$\pi_v$ is unramified.

\begin{lemma} \label{lem-important} 
With the previous notations, there exists a finite level $I \in \IC$ with
$I_v \simeq \Iw_v(h_0)$, and a short exact sequence
$$0 \rightarrow \Gamma_{\xi,\chi_v,\mathfrak m}(I,h_0+1) 
\longrightarrow \Gamma_{\xi,\chi_v,!,\mathfrak m}
(I,h_0+1,h_0) \longrightarrow   T \rightarrow 0$$
where $T \neq (0)$ and every irreducible sub-quotient of its $l$-torsion as a 
$\Tm_{\xi,\mathfrak m}^S \otimes_{\overline \Zm_l} \overline \Fm_l$-module, 
can be obtained as a sub-quotient of the torsion submodule of 
the cokernel of 
\addtocounter{thm}{1}
\begin{multline} \label{eq-torsion-trop2}
H^{0}(\sh_{I,\bar s_v},j^{=h_0+1}_! 
HT_\xi(\chi_v,\st_{h_0+1}(\chi_v)))_{\mathfrak m} \\
\longrightarrow H^0(\sh_{I,\bar s_v},j^{=h_0}_! 
HT_\xi(\chi_v,\st_{h_0}(\chi_v)))_{\mathfrak m}.
\end{multline}
\end{lemma}

\begin{proof}
The idea is to compute the cohomology of $\gr^{h_0}_{!}(\Psi_{\xi,\chi_v})$
in two different ways, first through the spectral sequence associated to
(\ref{eq-resolution-psi}) and secondly through its filtration of stratification with graded parts the Harris-Taylor perverse sheaves.

To argue we will rest on the level of non degeneracy at $v$ so that we pass to
$I^v(\oo)$-level: as $q_v$ modulo $l$ is of order $>d$ taking
invariant under $GL_d(\OC_v)$ is then an exact functor. 
First note that the $I^v(\oo)$-version of 
(\ref{eq-torsion-trop2}) is non strict 
if and only if the same is true for its non induced version in the next formula,
whatever is $\Pi_{h_0}$ a representation of $GL_{h_0}(F_v)$
\addtocounter{thm}{1}
\begin{multline} \label{eq-torsion-trop3}
H^{0}(\sh_{I^{v}(\oo),\bar s_v,\overline{1_{h_0}}},j^{=h_0+1}_{\overline{1_{h_0}},!} HT_{\overline{1_{h_0}},\xi}
(\chi_v,\Pi_{h_0}  \otimes \chi_v ))_{\mathfrak m} \\
\longrightarrow H^0(\sh_{I^{v}(\oo),\bar s_v,\overline{1_{h_0}}},j^{=h_0}_{\overline{1_{h_0}},!} 
HT_{\overline{1_{h_0}},\xi}(\chi_v,\Pi_{h_0}))_{\mathfrak m},
\end{multline}
where we denote by $\sh^{=h_0+1}_{I^v(\oo),\bar s_v,\overline{1_{h_0}}}$ the
disjoint union of the pure strata, cf. notation \ref{nota-strata}, 
$\sh^{=h_0+1}_{I^v(\oo),\bar s_v,g}$
contained in $\sh^{\geq h_0}_{I^v(\oo),\bar s_v,\overline{1_{h_0}}}$. As usual
the notation $j^{=h_0+1}_{\overline{1_{h_0}}}$ designates the closed embedding
of $\sh^{=h_0+1}_{I^v(\oo),\bar s_v,\overline{1_{h_0}}}$ in $\sh^{\geq 1}_{I^v(\oo),\bar s_v}$.
Recall that $h_0$ in \S \ref{para-max}, is chosen so that, as a $\overline \Zm_l[P_{h_0,d}(F_v)]$-module, for $\Pi_{h_0}$ unramified, using also 
the fact that $q_v$  modulo $l$ is of order $>d$ so that the functor of $P_{h_0,d}(\OC_v)$-invariants is exact, then the cokernel of (\ref{eq-torsion-trop3})
has non trivial vectors invariant under $P_{h_0,d}(\OC_v)$. We then deduce
the following facts, cf. also the remark after \ref{nota-iwh0}.

\begin{lemma} \label{lem-facts0}
\begin{itemize}
\item With $\Pi_{h_0}=\st_{h_0}(\chi_v)$, the cokernel of the induced version of
(\ref{eq-torsion-trop3}) has non trivial vectors invariant 
under $\Iw_v(h_0)$.

\item For $h > h_0$, whatever are the representations $\Pi_h$ and
$\Pi_{h+1}$ of respectively $GL_h(F_v)$ and $GL_{h+1}(F_v)$, the cokernel of
\addtocounter{thm}{1}
\begin{multline} \label{eq-torsion-trop-h}
H^{0}(\sh_{I^v(\oo),\bar s_v},j^{=h+1}_! 
HT_\xi(\chi_v,\Pi_{h+1} ))_{\mathfrak m} \\
\longrightarrow H^0(\sh_{I^v(\oo),\bar s_v},j^{=h}_! 
HT_\xi(\chi_v,\Pi_h))_{\mathfrak m},
\end{multline}
does not have non zero invariant vector under $GL_d(\OC_v)$.

\item For $\Pi_h=LT_{\chi_v}(h_0-1,h-h_0)$ and 
$\Pi_{h+1}=LT_{\chi_v}(h_0-1,h-h_0+1)$, as in (\ref{eq-map-hmax}),
the cokernel of the previous point does not have non zero invariant vector under
$\Iw_v(h_0)$.
\end{itemize}
\end{lemma}

We then compute the $\mathfrak m$-localized cohomology of 
$\gr^{h_0}_{!,\chi_v}(\Psi_{\varrho,\xi})$
in level $I^v(\oo)$ having non trivial invariant under $I$ with 
$I_v \simeq \Iw_v(h_0)$. 
By maximality of $h_0$, note that for $h_0<t \leq d$, the cohomology groups of
$P_\xi(t,\chi_v)$ and $j^{=t}_! HT_\xi(\pi_v,LT_{\chi_v}(h_0-1,t-h_0))$,
after localization by $\mathfrak m$, do no
have non trivial torsion vector invariant under $\Iw_v(h_0)$ as explained in
the previous lemma.

(1) Following the proof of \ref{lem-torsion-min} with the spectral sequence 
associated to 
(\ref{eq-resolution-psi}) and neglecting torsion classes
which do not have non trivial vectors invariants by $\Iw_v(h_0)$, 
we then deduce that the
$H^i_{tor}(\sh_{I^v(\oo),\bar s_v},
\gr^{h_0}_{!}(\Psi_{\xi,\chi_v}))_{\mathfrak m}$ do not have
non trivial vector invariant under $\Iw_v(h_0)$ if $i \neq 0,1$
while for $i=0$ the torsion is non trivial and the
vectors invariant by $\Iw_v(h_0)$ are given by the non strictness of 
\addtocounter{thm}{1}
\begin{multline} \label{eq-torsion-trop}
H^{0}(\sh_{I^{v}(\oo),\bar s_v},j^{=h_0+1}_! HT_\xi(\chi_v,LT_{\chi_v}(h_0-1,1)))_{\mathfrak m} \\
\longrightarrow H^0(\sh_{I^{v}(\oo),\bar s_v},j^{=h_0}_! 
HT_\xi(\chi_v,\st_{h_0}(\chi_v)))_{\mathfrak m}.
\end{multline}

(2) Concerning $H^0(\sh_{I^{v}(\oo),\bar s_v},
P_\xi(\chi_v,h_0))_{\mathfrak m}$, 
its torsion submodule is
parabolicaly induced, so that beside those coming from the non strictness of (\ref{eq-torsion-trop}),
there is also the contribution given by the non strictness of (\ref{eq-torsion-trop2}), which 
contains in particular a subquotient, denoted $\widetilde T$, 
such that $\widetilde T[l]$ is of level of non degeneracy strictly greater than those appearing 
in (\ref{eq-torsion-trop}). Note moreover that $\widetilde T[l]$
has non trivial vectors under $\Iw_v(h_0)$.

(3) Consider then the cohomology of $\gr^{h_0}_{!,\chi_v}(\Psi_{\varrho,\xi})$  
computed through
its filtration of stratification with graded parts, up to Galois shifts, the 
$H^0(\sh_{I^{v}(\oo),\bar s_v},P_\xi(h_0+k,\chi_v))_{\mathfrak m}$, 
for $0 \leq k \leq d-h_0$,
and more particularly the induced filtration of the free quotient of 
$H^0(\sh_{I^{v}(\oo),\bar s_v},\gr^{h_0}_{!,\chi_v}(\Psi_{\varrho,\xi}))_{\mathfrak m}$  as before.
As the level of non degeneracy of $\widetilde T[l]$ is higher than those of the $l$-torsion
of $H^0(\sh_{I^{v}(\oo),\bar s_v},\Fil^{h_0}_{!,\chi_v}(\Psi_{\varrho,\xi}))_{\mathfrak m}$, 
computed through the spectral sequence associated to (\ref{eq-resolution-psi}),
we then have a filtration of this free quotient where both appears
\begin{itemize}
\item torsion modules such as $\widetilde T$,

\item and the free sub-quotients are given by the lattices $\Gamma_{\xi,\chi_v,\mathfrak m}(I^v,h_0+i)$ 
of the free quotient of the localized cohomology of $P_\xi( \chi_v,h_0+i)$ 
for $0 \leq i \leq d-h_0$.
\end{itemize}
We know go back to the level $I=I^v\Iw_v(h_0)$: as $q_v$ modulo $l$ is of order strictly
greater than $d$, the functor of $\Iw_v(h_0)$-invariants is exact. As only contributes the
cohomology of $P_\xi( \chi_v,h_0+i)$ for $i=0,1$ the result follows from
the fact that $\widetilde T$ has non trivial invariant under $\Iw_v(h_0)$.
\end{proof}

Recall that 
$$\gr^{h_0}_!(\Psi_\varrho) \otimes_{\overline \Zm_l} \overline \Qm_l \simeq 
\bigoplus_{\chi_v \in \cusp_\varrho} \gr^{h_0}_{!}(\Psi_{\chi_v})$$
so that we can find a filtration of $\gr^{h_0}_!(\Psi_\varrho)$ whose graded parts are
free and isomorphic, after tensoring with $\overline \Qm_l$, to 
$\gr^{h_0}_{!}(\Psi_{\chi_v})$.
Arguing as in the proof of lemma \ref{lem-torsion-min}, using (\ref{eq-map-hmax}), 
we have the following result.

\begin{lemma} \label{lem-important2}
For every $1 \leq t$, let $j(t)$ be the minimal integer $j$ such that the torsion of
$H^j(\sh_{I^{v}(\oo),\bar s_v},\gr^t_{!}(\Psi_{\xi,\varrho}))_{\mathfrak m}$ has 
non trivial invariant vectors under $\Iw_v(h_0)$. Then
$$j(t)=\left \{ \begin{array}{ll}
+ \oo & \hbox{if } t \geq h_0+1, \\
t-h_0 & \hbox{for } 1 \leq t \leq h_0.
\end{array} \right.$$
Moreover as a $\Tm_{\xi,\mathfrak m} \otimes_{\overline \Zm_l} \overline \Fm_l$-module,
up to multiplicities, the irreducible sub-quotients of 
$H_{tor}^{j(t)}(\sh_{I^{v}\Iw_v(h_0),\bar s_v},\gr^t_{!}(\Psi_{\xi,\varrho}))_{\mathfrak m}$ are
independent of $t$.
\end{lemma}

\marque \noindent \textit{Important fact}:
Note that, up to multiplicities, the irreducible $\Tm_{\xi,\mathfrak m}
\otimes_{\overline \Zm_l} \overline \Fm_l$-sub-quotients\footnote{i.e. if one forget
the action of $GL_d(F_v)$}
of the $l$-torsion of the cohomology of $\gr^{h_0}_{!,\chi_v}(\Psi_{\varrho,\xi})$ and
$P_\xi(h_0,\chi_v)$ in level $I^v\Iw_v(h_0)$ 
are the same, given by the non strictness of the 
maps (\ref{eq-map-hmax}).

The idea is now to increase the level at another place $w \in \spl(I)$ verifying the same
hypothesis than $v$, i.e. $q_w$ modulo $l$ is of order strictly greater than $d$.
When the level at $w$ is infinite, i.e. for $I^w(\oo)$ with $I_v \simeq \Iw_v(h_0)$, using
again the exactness of invariant with $GL_d(\OC_w)$, from the previous observation we
deduce that considering irreducible $\overline \Fm_l$-representations of
$GL_d(F_w)$, the sets, i.e. without multiplicities, of irreducible sub-quotients of the $l$-torsion
of the cohomology of respectively $P_\xi(h_0,\chi_v)$ and 
$\gr^{h_0}_{!,\chi_{v}}(\Psi_{\varrho,\xi})$, are the same.

\begin{prop} \label{prop-important}
Up to multiplicities, the set of irreducible
$\overline \Fm_l[GL_d(F_w)]$-sub-quotients of the $l$-torsion of\footnote{or
those of $H^{0}(\sh_{I^w(\oo),\bar s_v},P_\xi(h_0,\chi_v))_{\mathfrak m}$
as explained above}
$H^{0}(\sh_{I^w(\oo),\bar s_v},\gr^{h_0}_{!}(\Psi_{\xi,\varrho}))_{\mathfrak m}$,
are the same as those of
$H^{d-h_0}(\sh_{I^w(\oo),\bar s_v},V_{\xi,\overline \Zm_l})_{\mathfrak m}$.
\end{prop}

\begin{proof}
We compute $H^{d-h_0(I^v)}(\sh_{I^w(\oo),\bar s_v},V_{\xi,\overline \Zm_l})_{\mathfrak m}$
using the filtration $\Fil^\bullet_!(\Psi_{\xi,\varrho})$ through the spectral sequence 
(\ref{eq-ss2}). Recall that for every $p+q \neq 0$, the free quotient of $E^{p,q}_{!,\varrho,1}$
are zero. By definition of the filtration these $E^{p,q}_{!,\varrho,1}$ are trivial for $p \geq 0$
while, thanks to the previous lemma, for any $p \leq -1$ there are zero for
$p+q < j(p):=p-h_0$. Note then that
$E_{!,\varrho,1}^{-1,j(1)+1}$ which is torsion and non zero, according to the previous lemma,
is equal to $E_{!,\varrho,\oo}^{j(1)} \simeq
H^{d-h_0}(\sh_{I^w(\oo),\bar s_v},V_{\xi,\overline \Zm_l})_{\mathfrak m}$.
\end{proof}

Consider as before $I^w(\oo)$ such that its local component at $v$ is 
$\Iw_v(h_0)$.
Then combining the result of lemma \ref{lem-important} in level $I^w(\oo)$, with the previous
proposition, we then deduce that the cokernel $T$ of \ref{lem-important} verifies
the following property. As a $\overline \Fm_l$-representation of $GL_d(F_w)$, every irreducible
sub-quotient of $T[l]$ is also a sub-quotient of
$H^{d-h_0}(\sh_{I^w(\oo),\bar s_v},V_{\xi,\overline \Zm_l})_{\mathfrak m}$.
Then applying lemma \ref{lem-generic-sq} at the place $w$ playing a symmetric role
as $v$, we then deduce the following result.

\begin{corol} \label{coro-absurd}
As a $\overline \Fm_l$-representation of $GL_d(F_w)$, the $l$-torsion of 
 the cokernel $T$ of lemma \ref{lem-important} in level $I^w(\oo)$ as above,
does not contain any irreducible generic sub-quotient with cuspidal support made of characters. 
\end{corol}

- We can now repeat the arguments with $\gr^k_{!,\chi_v}(\Psi_{\varrho,\xi})$ for any
$1 \leq k \leq h_0$. More precisely, cf. the last remark of \S \ref{para-fil-psi}, 
consider $\Fil^{i}(\gr^{k}_{!,\chi_c}(\Psi_{\varrho,\xi}))$
for $i=h_0-k$ and $i=h_0-k+2$. Take a level 
$$I=I^{v,w}I_w(\oo) \Iw_v(h_0),$$ 
infinite at $w$ and Iwahori at $v$,
such that, arguing by absurdity, the torsion of 
$H^{d-1}(\sh_{I^{v,w} GL_d(\OC_w \otimes \OC_v),\bar \eta},V_{\xi,\overline \Zm_l})_{\mathfrak m}$
is non trivial\footnote{For the definition of $h_0:=h_0(I^v)$, cf. notation \ref{nota-h0}.}.

There exists $\widetilde{\mathfrak m} \subset \mathfrak m$ such that
$$
\Pi_{\widetilde{\mathfrak m}} \simeq
\st_{h_0+1}(\chi_v) \times \chi_{v,1} \times \cdots \times \chi_{v,d-h_0-1},$$
with $\chi_v \equiv \varrho \mod l$. 

%
%
%
%

Moreover as before
\begin{itemize}
\item $\Fil^{h_0+2-k}(\gr^{k}_{!,\chi_v}(\Psi_{\varrho,\xi}))$
has trivial cohomology groups in level $I$ because $I_v=\Iw_v(h_0)$ and
the irreducible constituents of 
$\Fil^{h_0+2-k}(\gr^{k}_{!,\chi_v}(\Psi_{\varrho,\xi})) 
\otimes_{\overline \Zm_l} \overline \Qm_l$ are, up to Galois shift,
Harris-Taylor perverse
sheaves $P(t,\chi_v)$ with $t \geq h_0+2$;

\item we can apply the previous argument relatively to 
$\gr^{h_0}_{!,\chi_v}(\Psi_{\xi,\varrho})$ to the 
quotient $Q:=\Fil^{h_0-k}(\gr^{k}_{!,\chi_v}(\Psi_{\varrho,\xi})) 
/\Fil^{h_0+2-k}(\gr^{h_0}_{!,\chi_v}(\Psi_{\varrho,\xi}))$, so that, denoting by
$\Gamma'_{\xi,\chi_v,!,\mathfrak m} (I,h_0+1,k)$
the lattice of (\ref{eq-def-sq}) given by the free quotient of 
$H^0(\sh_{I^w(\oo),\bar s_v},Q)_{\mathfrak m}$, the cokernel $T'_k$ of
$$0 \rightarrow \Gamma_{\xi,\chi_v,\mathfrak m}(I,h_0+1) 
\longrightarrow \Gamma'_{\xi,\chi_v,!,\mathfrak m}
(I,h_0+1,k) \longrightarrow   T'_k \rightarrow 0$$
is such that $T'_k[l] \neq (0)$ and, as a $\overline \Fm_l$-representation of 
$GL_d(F_w)$, it does not contain any irreducible generic sub-quotient made of 
characters.

\item In addition of the previous arguments, we also have to deal with the
torsion in the cohomology groups of 
$$\gr^{k}_{!,\chi_v}(\Psi_{\varrho,\xi}) /\Fil^{h_0-k}(\gr^{k}_{!,\chi_v}(\Psi_{\varrho,\xi})),$$ 
which could modify the lattice $\Gamma'_{\xi,\chi_v,!,\mathfrak m}
(I,h_0+1,k)$ to give the good one denoted above by 
$\Gamma_{\xi,\chi_v,!,\mathfrak m}(I,h_0+1,k)$ . Note
again that,
as a $\overline \Fm_l$-representation of $GL_d(F_w)$, this $l$-torsion
does not contain any
irreducible generic sub-quotient made of characters, so the cokernel of 
$$\Gamma'_{\xi,\chi_v,!,\mathfrak m} (I,h_0+1,k)
\hookrightarrow \Gamma_{\xi,\chi_v,!,\mathfrak m}(I,h_0+1,k),$$
is again such that,  as a $\overline \Fm_l$-representation of 
$GL_d(F_w)$, its $l$-torsion does not contain any irreducible generic 
sub-quotient made of characters.
\end{itemize}
Forgetting again Galois shifts, we then conclude that the $l$-torsion of
the cokernel $T_k$ of
\addtocounter{thm}{1}
\begin{equation} \label{eq-absurd}
0 \rightarrow \Gamma_{\xi,\chi_v,\mathfrak m}(I,h_0+1) 
\longrightarrow \Gamma_{\xi,\chi_v,!,\mathfrak m}
(I,h_0+1,k) \longrightarrow   T_k \rightarrow 0,
\end{equation}
is non zero and,
as a $\overline \Fm_l$-representation of $GL_d(F_w)$, it
does not contain any
irreducible generic sub-quotient made of characters.

\medskip

- We now compute
$H^{d-1}(\sh_{I,\bar \eta},V_{\xi,\overline \Zm_l})_{\mathfrak m}$ through the spectral
sequence of vanishing cycles using the filtration
$$\Fil^{1}_!(\Psi_\varrho) \harrow \cdots \harrow \Fil^{d-1}_!(\Psi_\varrho) \harrow 
\Fil^{d}_!(\Psi_\varrho) \harrow \Psi_\varrho,$$
and with level $I=I^{v,w}I_w(\oo) \Iw_v(h_0)$ 
infinite at $w$ and Iwahori at $v$. Recall that we can filtrate each of the
$\gr^k_!(\Psi_\varrho)$ as follows:
\begin{itemize}
\item Define first $\Fil^\bullet(\gr^k_!(\Psi_\varrho))$ such that
$$\gr^r(\gr^k_!(\Psi_\varrho)) \otimes_{\overline \Zm_l} \overline \Qm_l
\simeq \bigoplus_{\pi_v \in \cusp_r(\varrho)} P(\pi_v,t),$$
where $t$ is such that $k=tg_r(\varrho)$.

\item Then each of the previous $\gr^r(\gr^k_!(\Psi_\varrho))$ has a filtration
with graded parts such that over $\overline \Qm_l$ we recover, up to Galois shift, 
$P(\pi_v,t)$.
\end{itemize}

\rem As we choose the level Iwahori at $v$, when computing the cohomology
we are only concerned with characters $\chi_v \in \cusp_{-1}(\varrho)$.

As before, arguing by absurdity, 
we suppose that the torsion of 
$H^{d-1}(\sh_{I^{v,w} GL_d(\OC_w \otimes \OC_v),\bar \eta},V_{\xi,\overline \Zm_l})_{\mathfrak m}$
is non trivial\footnote{For the definition of $h_0:=h_0(I^v)$, cf. notation \ref{nota-h0}.}, and we pay special attention to the lattices of 
\addtocounter{thm}{1}
\begin{equation} \label{eq-def-Vrho}
V_{\xi,\chi_v,\mathfrak m}(I,h_0+1)(\delta):=
H^i(\sh_{I,\bar s},P_\xi(h_0+1,\chi_v))(\delta)_{\mathfrak m} 
\otimes_{\overline \Zm_l} \overline \Qm_l,
\end{equation}
for $\chi_v \in \cusp_{-1}(\varrho)$ and various $\delta$.

We first start with $\delta=-h_0/2$. Note that the 
$H^i(\sh_{I,\bar s_v},\Psi_{\xi,\varrho}/\Fil^{h_0+1}_!(\Psi_{\xi,\varrho}))_{\mathfrak m}$
are all zero. Indeed the torsion free graded parts $\grr^k(\Psi_\varrho)$ 
of any exhaustive filtration of 
$\Psi_\varrho / \Fil^{h_0+1}_!(\Psi_\varrho)$, up to Galois torsion, are such that
$\grr^k(\Psi_\varrho) \otimes_{\overline \Zm_l} \overline \Qm_l \simeq P(t,\pi_v)$ 
whith $\pi_v \in \cusp_\varrho$
an irreducible cuspidal representation of some $GL_g(F_v)$ with $tg >h_0+1$. Then
every irreducible constituant of 
$H^i(\sh_{I^v(\oo),\bar s_v},\grr^k(\Psi_{\xi,\varrho}))_{\mathfrak m} 
\otimes_{\overline \Zm_l} \overline \Fm_l$,
as a $\overline \Fm_l$-representation of $GL_d(F_v)$ is a sub-quotient of an induced 
representation
$r_l(\st_t(\pi_v))\{ \delta/2 \} \times \tau$ for some irreducible $\overline \Fm_l$-representation
$\tau$ of $GL_{d-tg}(F_v)$. In particular such a representation does not have non trivial
invariants under $\Iw_v(h_0)$, so that, as the functor of $\Iw_v(h_0)$-invariants is exact,
there is no cohomology in level $I$ as stated.

We then deduce that
$H^0(\sh_{I,\bar s},P_{\xi,\overline \Qm_l}(h_0+1,\chi_v))(\delta)_{\mathfrak m}.$
is a quotient of  
$H^0(\sh_{I,\bar s_v},\Psi_{\xi,\varrho})_{\mathfrak m} 
\otimes_{\overline \Zm_l} \overline \Qm_l$
and we denote by $\Gamma_{\xi,\chi_v,\Psi,\mathfrak m}(I,h_0+1,+)$ 
its stable lattice induced by 
$H^0_{free}(\sh_{I,\bar s_v},\Psi_{\xi,\varrho})_{\mathfrak m}$.

\begin{prop} \label{prop-lattice-psi}
With the previous notations, we have
$$0 \rightarrow \Gamma_{\xi,\chi_v,\Psi,\mathfrak m}(I,h_0+1,+)
\longrightarrow \Gamma_{\xi,\chi_v,\mathfrak m}(I,h_0+1) \longrightarrow T 
\rightarrow 0,$$
where $T[l]$, as a $\overline{\mathbb F}_l$-representation of $GL_d(F_w)$,
does not have any irreducible generic sub-quotient with cuspidal support
made of characters.

\end{prop}

\begin{proof}
We compute $H^0(\sh_{I,\bar s_v},\Psi_{\xi,\varrho})_{\mathfrak m}$ through 
the spectral sequence associated to the filtration
$$\Fil^{1}_!(\Psi_{\xi,\varrho}) \harrow \cdots \harrow \Fil^{d-1}_!(\Psi_{\xi,\varrho})
\harrow  \Fil^{d}_!(\Psi_{\xi,\varrho}) \harrow \Psi_{\xi,\varrho}.$$
Recall that $H^0(\sh_{I,\bar s_v},\Psi_{\xi,\varrho})_{\mathfrak m}=
H^0(\sh_{I,\bar s_v},\Fil_!^{h_0+1}(\Psi_{\xi,\varrho})_{\mathfrak m}$.
Let denote
$$K_0 =\ker (\Fil^{h_0+1}_!(\Psi_{\xi,\varrho}) \twoheadrightarrow 
P_\xi(h_0+1,\chi_v)(-\frac{h_0}{2})).$$
Recall that over $\overline \Qm_l$, all the cohomology groups are concentrated
in degree zero. We then have
$$0 \rightarrow \Gamma_{\xi,\chi_v,\mathfrak m}(I,h_0+1,+) \longrightarrow
 \Gamma_{\xi,\chi_v,\Psi,\mathfrak m}(I,h_0+1)
\longrightarrow  T \rightarrow 0,$$
where $T \hookrightarrow H^1_{tor}(\sh_{I,\bar s_v}, K_0)_{\mathfrak m}$.
The statement about $T[l]$ then follows from the previous section.

\end{proof}

%

Consider now the case $\delta=h_0/2$ and denote as before by
$\Gamma_{\xi,\chi_v,\Psi,\mathfrak m}(I,h_0+1,-)$ the lattice of 
$V_{\xi,\chi_v,\mathfrak m}(I,h_0+1)(h_0/2)$ induced by the free quotient
of $H^0(\sh_{I,\bar s_v},\Psi_{\varrho,\xi})_{\mathfrak m}$.
Consider $\Fil^1_{!,\chi_v}(\Psi_{\varrho,\xi}) \harrow \Fil^1_!(\Psi_{\varrho,\xi})
\harrow \Psi_{\varrho,\xi}$.

\rem Before $\Fil^1_{!,\chi_v}(\Psi_{\varrho,\xi})$ was defined as a quotient of
$\Fil^1_!(\Psi_{\varrho,\xi})$. To separate the $\chi_v \in \cusp_{-1}(\varrho)$
in $\Fil^1_!(\Psi_{\varrho,\xi})$ we start with a filtration of
$j^{=1,*} \Fil^1_!(\Psi_{\varrho,\xi})$ with
$$j^{=1,*} \Fil^1_!(\Psi_{\varrho,\xi}) \bigoplus_{\chi_v \in \cusp_{-1}(\varrho)}
HT_\xi(\chi_v,\chi_v),$$
by considering an numbering of $\cusp_{-1}(\varrho)$. In particular for a fixed
$\chi_v$, when it appears in first or last position, we may to modify the lattice
of $HT(\chi_v,\chi_v)$. But as, up to homothety, $HT(\chi_v,\chi_v)$ has an unique
stable lattice, we then obtain the same $\Fil^1_{!,\chi_v}(\Psi_{\varrho,\xi})$.

We then has seen that the lattice $\Gamma_{\xi,\chi_v,!,\mathfrak m}(I,h_0+1,1)$
is such that
\addtocounter{thm}{1}
\begin{equation} \label{eq-lattice2}
0 \rightarrow \Gamma_{\xi,\chi_v,\mathfrak m}(I,h_0+1) 
\longrightarrow \Gamma_{\xi,\chi_v,!,\mathfrak m}
(I,h_0+1,1) \longrightarrow   T \rightarrow 0,
\end{equation}
where $T \neq (0)$ is such that $T[l]$, as a $\overline \Fm_l$-representation of
$GL_d(F_w)$, does not contain any
irreducible generic sub-quotient made of characters.
We then conclude by observing that
$$\Gamma_{\xi,\chi_v,\Psi,\mathfrak m}(I,h_0+1,-)=
\Gamma_{\xi,\chi_v,!,\mathfrak m}(I,h_0+1,1).$$

\rem From proposition \ref{prop-lattice-psi} and (\ref{eq-lattice2}),
as the modulo $l$ reduction of 
$\Gamma_{\xi,\chi_v,\Psi,\mathfrak m} (I,h_0+1,-)$
has a generic sub-quotient, it is not isomorphic to 
$\Gamma_{\xi,\chi_v,\Psi,\mathfrak m} (I,h_0+1,+)$.

\subsection{Global lattices and generic representations}
\label{para-final}

%

We argue on the set $S_{\mathfrak m}(v)$ of modulo $l$ eigenvalues of 
$\overline \rho_{\mathfrak m}(\frob_v)$. 
By hypothesis there exists $\lambda:=\chi_v(\varpi_v)$ such that 
$$\{ \lambda q_v^{h_0/2},\lambda q_v^{h_0/2-1},\cdots,\lambda q_v^{-h_0/2} \} 
\subset S_{\mathfrak m}(v),$$ 
where $h_0 \geq 1$ is defined before under the hypothesis there exists non
trivial torsion. More precisely $\lambda q_v^{h_0/2}$ is the modulo $l$ reduction
of the eigenvalue of $\frob_v$ acting on $P_\xi(\chi_v,h_0+1)(\frac{h_0}{2})$
such that the torsion of $H^0(\sh_{I,\bar s_v},P_\xi(\chi_v,h_0)(\frac{h_0}{2}))_{\mathfrak m}$
is non zero.

We then start with any $\lambda_0 q_v^{h_0/2} \in S_{\mathfrak m}(v)$ and
we want to prove that $S_{\mathfrak m}(v)$ contains a subset 
$\{ \lambda_1 q_v^{h_0/2},\lambda_0 q_v^{h_0/2-1},\cdots,\lambda_1 
q_v^{-h_0/2} \}$ such that
\begin{itemize}
\item there exists $0 < r < h_0+1$ with 
$\lambda_1 q_v^{h_0/2}=\lambda_0 q_v^{h_0/2+r}$,

\item and $\lambda_1 q_v^{h_0/2}$ is the modulo $l$ reduction of the eigenvalue 
of $\frob_v$ acting on $P_\xi(\chi'_v,h_0+1)(\frac{h_0}{2})$ such that the torsion of 
$H^0(\sh_{I,\bar s_v},P_\xi(\chi'_v,h_0)(\frac{h_0}{2}))_{\mathfrak m}$
is non zero.
\end{itemize}
We then obtain another interval inside $S_{\mathfrak m}(v)$ 
containing strictly the previous one and
we can then play again with $\lambda_1 q_v^{h_0}$ and repeat the above
property. At the end we then obtain the full set $\{ \lambda_0 q_v^n: n \in \Zm \}$
which is of order the order of $q_v$ modulo $l$. But this order is by hypothesis
strictly greater than $d$ although trivially the set $S_{\mathfrak m}(v)$ 
of eigenvalues of $\rho(\frob_v)$ is of order $\leq d$.

We now explain how to increase the interval as stated above. To do so
start first with a classical fact concerning $\overline \Zm_l[G]$-modules where
$G$ is a group. Let then $\Gamma$, $\Gamma_1$ and $\Gamma_2$ three 
$\overline \Zm_l$-free modules with an action of a group $G$ such that
\addtocounter{thm}{1}
\begin{equation}\label{eq-extension}
0 \rightarrow \Gamma_1 \longrightarrow \Gamma \longrightarrow \Gamma_2 
\rightarrow 0,
\end{equation}
which is $G$ equivariant. We then suppose that this extension is split
over $\overline \Qm_l$, i.e. 
$$\Gamma \otimes_{\overline \Zm_l} \overline \Qm_l
\simeq (\Gamma_1 \otimes_{\overline \Zm_l} \overline \Qm_l ) \oplus
(\Gamma_2 \otimes_{\overline \Zm_l} \overline \Qm_l).$$
Let then denote by 
$\Gamma'_2:=\Gamma \cap (\Gamma_2 \otimes_{\overline \Zm_l} \overline \Qm_l)$
and $\Gamma'_1:=\Gamma/\Gamma_2'$ so that
$$0 \rightarrow \Gamma'_2 \longrightarrow \Gamma \longrightarrow \Gamma'_1 
\rightarrow 0.$$
We then have the following commutative diagram
\addtocounter{thm}{1}
\begin{equation} \label{eq-prop-extension}
\xymatrix{
 & \Gamma_1 \ar@{^{(}->}[d] \ar@{=}[r] & \Gamma_1 \ar@{^{(}->}[d] \\
\Gamma'_2 \ar@{^{(}->}[r] \ar@{=}[d] & \Gamma \ar@{->>}[r] \ar@{->>}[d] & 
\Gamma'_1 \ar@{->>}[d]  \\ \Gamma'_2 \ar@{^{(}->}[r] & 
\Gamma_2 \ar@{->>}[r] & T,  \\
}
\end{equation}
where $T$ is of torsion and zero if and only if (\ref{eq-extension}) is split, i.e.
$\Gamma'_1=\Gamma_1$ and $\Gamma'_2=\Gamma_2$.

\noindent \textit{Important remark}:
in the following we will consider $\Tm_{I,\xi,\mathfrak m}[\gal_{F,S}]$-free modules
when $\Gamma_i \otimes_{\overline \Zm_l} \overline \Fm_l$ for $i=1,2$
are isotypic for the Galois action relatively to a character $\overline \chi_i$
such that $\overline \chi_1 \not \simeq \overline \chi_2$. From the
previous diagram, then $T=0$ and $\Gamma \simeq \Gamma_1 \oplus \Gamma_2$.

The idea is to consider two distincts filtrations 
$\Gamma:=H^0_{free}(\sh_{I,\bar s_v},\Psi_{\xi,\varrho})_{\mathfrak m}$,
and use the previous diagram to explain that, upon the hypothesis that the
torsion is non zero, we could not go through one to the other.
Recall first that
$$\Gamma \otimes_{\overline \Zm_l} \overline \Qm_l \simeq
\bigoplus_{\widetilde{\mathfrak m}} \Pi_{\widetilde{\mathfrak m}}^I \otimes 
\rho_{\widetilde{\mathfrak m}},$$
so that, by fixing any numbering 
$$\{ \widetilde{\mathfrak m} \subset \mathfrak m \hbox{ s.t. } 
\Pi_{\widetilde{\mathfrak m}}^I  \neq (0) \}=
\{ \widetilde{\mathfrak m}_1,\cdots, \widetilde{\mathfrak m}_r \},$$
we define a filtration $\Fil^\bullet(\Gamma)$ with graded parts
$\gr^k(\Gamma)$ which is a stable lattice of
$\Pi^I_{\widetilde{\mathfrak m}_k} \otimes \rho_{\widetilde{\mathfrak m}_k}$.
With previous notation, we suppose that $\Pi_{\widetilde{\mathfrak m}_r,v}
\simeq \st_{h_0+1}(\chi_v) \times \chi_{v,1} \times \cdots \times \chi_{v,d-h_0-1}$
with $\chi_v \in \cusp_{-1}(\varrho)$. As the modulo $l$ reduction of
$\rho_{\widetilde{\mathfrak m}_r}$ is irreducible, then 
$\gr^r(\Gamma)$ is typic, in the sense of \cite{scholze-LT} \S 5, i.e.
$$\gr^r(\Gamma) \simeq \Gamma_r \otimes \Lambda_r,$$
where $\Gamma_r$ is a stable lattice of $\Pi_{\widetilde{\mathfrak m}_r}$ on
which the Galois action is trivial and $\Lambda_r$ is a stable lattice of
$\rho_{\widetilde{\mathfrak m}_r}$.

We now start from the filtration of $\Gamma$ induced by the filtration
$\Fil^\bullet_!(\Psi_\varrho)$ and we try, using diagrams like above, to end up
to the previous lattice $\gr^r(\Gamma) \simeq \Gamma_r \otimes \Lambda_r$.
As explained in the previous section, $\Gamma_{\xi,\chi_v,\Psi,\mathfrak m}
(I,h_0+1,+)$ is a quotient of $\Gamma$ and so a quotient of $\gr^r(\Gamma)$
so that 
$$\Gamma_r \simeq \Gamma_{\xi,\chi_v,\Psi,\mathfrak m} (I,h_0+1,+).$$
At the opposite, we see 
$\Gamma_{\xi,\chi_v,\Psi,\mathfrak m} (I,h_0+1,-)$
as a quotient of 
$H^0_{free}(\sh_{I,\bar s_v},\Fil^1_!(\Psi_{\xi,\varrho}))_{\mathfrak m}$.
As noticed in the last remark of the previous section, 
$\Gamma_{\xi,\chi_v,\Psi,\mathfrak m} (I,h_0+1,-)$
and $\Gamma_{\xi,\chi_v,\Psi,\mathfrak m} (I,h_0+1,+)$
are not isomorphic.
We then deduced that there should exists $\chi'_v \in \cusp_{-1}(\varrho)$
and a diagram \ref{eq-prop-extension} with 
$\Gamma'_2=\Gamma_{\xi,\chi_v,\Psi,\mathfrak m} (I,h_0+1,-)$
and $\Gamma'_1$ associated to a sub-quotient of $H^0(\sh_{I,\bar s_v},
P_\xi(\chi'_v,t)(\frac{t-1}{2}-\delta))_{\mathfrak m}$. Note that
as this $P_\xi(\chi'_v,t)(\frac{t-1}{2}- \delta)$ is a sub-quotient of some
$\gr^k_!(\Psi_{\xi,\varrho)}$ with $k \geq 2$, then we must have 
$2 \leq t \leq h_0+1$ and $\delta>0$.

In particular, from the previous important remark, 
we also deduce that $\chi'_v(\frac{t-1}{2}-\delta) \equiv
\chi_v(\frac{h_0}{2}) \mod l$ so that, by denoting $\lambda_1=\chi'_v(\varpi_v)$,
we can write $\lambda_1 q_v^{h_0/2}=\lambda_0 q_v^{h_0/2+r}$
with $0 < r <h_0+1$.

To be able to play again with $\chi'_v$ in place of $\chi_v$, we just need to prove
that the $\mathfrak m$-localized
cohomology groups of $P_\xi(\chi'_v,h_0)$ with level $I$, have torsion in degree
$0$ and $1$. It is equivalent to look at the cohomology of 
$P_\xi(\chi'_v,h_0)(\frac{h_0+1-t+2\delta}{2})=P_x(\chi_{v,1},h_0)$
with $\chi_{v,1}:=\chi'_v (\frac{h_0+1-t+2\delta}{2})$ verifying 
$\chi_{v,1} \equiv \chi_v \mod l$.
Recall the well-known short exact sequence
$$0 \rightarrow H^i(X,P) \otimes_{\overline \Zm_l} \overline \Fm_l \longrightarrow
H^i(X,P \otimes_{\overline \Zm_l} \overline \Fm_l) \longrightarrow
H^{i+1}(X,P)[l] \rightarrow 0.$$
We apply it to $X=\sh_{I,\bar s_v}$ and $P=P_\xi(\chi_{v,1},h_0)$, and we
recall that its $\overline \Qm_l$-cohomology localized at $\mathfrak m$,
is concentrated in degree $0$.
The same is true for $P_\xi(\chi_v,h_0)$ while it has torsion in degree $0$ and $1$
such that its $\overline \Fm_l$-cohomology localized at $\mathfrak m$, is 
concentrated in degree $-1,0,1$.
The same is then true for the $\overline \Fm_l$-cohomology, 
localized at $\mathfrak m$, of $P_\xi(\chi_{v,1},h_0)$
so that its $\overline \Zm_l$-cohomology localized at $\mathfrak m$, must have 
torsion in degree $0$ and $1$.

\bibliographystyle{plain}
\bibliography{bib-ok}

\end{document}